\documentclass[article,12pt]{amsart}

\usepackage{amsmath}
\usepackage{mathrsfs}
\usepackage{mathabx}
\usepackage{amsfonts}
\usepackage{graphicx,psfrag,epsf}
\usepackage{color}
\usepackage{epsfig}
\usepackage{enumerate}

\usepackage{subcaption}
\usepackage{natbib}
\RequirePackage[colorlinks,citecolor=blue,urlcolor=blue]{hyperref}

\usepackage{url} 

\numberwithin{equation}{section}
\theoremstyle{plain}
\newtheorem{thm}{Theorem}[section]
\newtheorem{rem}{Remark}[section]

\newtheorem{lem}{Lemma}[section]

\newcommand{\dE}{\mathbb{E}}
\newcommand{\dR}{\mathbb{R}}

\newcommand{\cA}{\mathcal{A}}

\newcommand{\cN}{\mathcal{N}}

\newcommand{\rI}{\mathrm{I}}
\newcommand{\cF}{\mathcal{F}}

\newcommand{\cM}{\mathcal{M}}
\newcommand{\veps}{\varepsilon}
\newcommand{\wh}{\widehat}

\newcommand{\skn}{\sum\limits_{k=1}^n}

\newcommand{\fn}{\widehat{f}_n(x)}
\newcommand{\fcn}{\widecheck{f}_n(x)}
\newcommand{\ftn}{\widetilde{f}_n(x)}
\newcommand{\fcnp}{\widecheck{f}_n^\prime(x)}
\newcommand{\ftnp}{\widetilde{f}_n^\prime(x)}
\newcommand{\gn}{\widehat{g}_n(x)}
\newcommand{\hn}{\widehat{h}_n(x)}
\newcommand{\fnp}{\widehat{f}_n^{\,\prime}(x)}
\newcommand{\gnp}{\widehat{g}_n^{\,\prime}(x)}
\newcommand{\gnc}{\widehat{g}_n^{\,2}(x)}
\newcommand{\hnp}{\widehat{h}_n'(x)}
\newcommand{\KXk}{K\left( \dfrac{x-X_k}{h_k}\right)}
\newcommand{\KXn}{K\left(\dfrac{x-X_n}{h_n}\right)}

\newcommand{\KpXk}{K^\prime\left(\dfrac{x-X_k}{h_k}\right)}
\newcommand{\KpXn}{K^\prime\left(\dfrac{x-X_n}{h_n}\right)}

\newcommand{\ind}{\mbox{1}\kern-.25em \mbox{I}}
\font\calcal=cmsy10 scaled\magstep1
\def\build#1_#2^#3{\mathrel{\mathop{\kern 0pt#1}\limits_{#2}^{#3}}}
\def\liml{\build{\longrightarrow}_{}^{{\mbox{\calcal D}}}}

\def\videbox{\mathbin{\vbox{\hrule\hbox{\vrule height1.4ex \kern.6em\vrule height1.4ex}\hrule}}}
\def\demend{\hfill $\videbox$\\}

\email{Bernard.Bercu@u-bordeaux.fr}
\email{Sami.Capderou@u-bordeaux.fr}
\email{Gilles.Durrieu@univ-ubs.fr}

\begin{document}
\title[Estimation of the derivative of the regression function]
{Nonparametric estimation of the derivative of the regression function: application to sea shores water quality 
 \vspace{1ex}}
 
\author{Bernard Bercu}
\address{Universit\'e de Bordeaux, Institut de Math\'ematiques de Bordeaux,
UMR CNRS 5251, 351 Cours de la Lib\'eration, 33405 Talence, France.}
\author{Sami Capderou}
\address{Universit\'e de Bordeaux, Institut de Math\'ematiques de Bordeaux,
UMR CNRS 5251, 351 Cours de la Lib\'eration, 33405 Talence, France.
}
\author{Gilles Durrieu}
\address{Universit\'e de Bretagne Sud, Laboratoire de Math\'ematiques de Bretagne Atlantique,
UMR CNRS 6205, Campus de Tohannic, 56017 Vannes, France.}
\thanks{}
 
 
\begin{abstract}
This paper is devoted to the nonparametric estimation of the derivative of the regression function
in a nonparametric regression model. We implement a very efficient and easy to handle
statistical procedure based on the derivative of the recursive Nadaraya-Watson estimator. 
We establish the almost sure convergence as well as the asymptotic normality for our estimates. 
We also illustrate our nonparametric estimation procedure on simulated and real life data 
associated with sea shores water quality and valvometry.
\end{abstract}

\maketitle


\section{Introduction}

Environmental and water protection should be tackled as a top priority of our society.
It is forecasted that in $2035$, nearly $60$\% of the world's population will live within $65$ miles of the sea front (\cite{haslett2001coastal}). 
Water quality monitoring is therefore fundamental especially on the coastline. On the one hand,  marine pollution comes 
mostly from land based sources. On the other hand, this pollution can lead to the collapse of coastal ecosystems and cause 
public health issues. In this context, there is a critical need to develop a real-time reliable field assay to monitor the water quality 
within a decision making process.  Among them, bioindicators are more and more commonly used.
Endemic species are the most suitable bioindicators for the assessment of the quality of the coastal environment.
For exemple,  oysters, a well-known filter-feeding mollusc, feature a relevant sentinel organism to evaluate water quality.
These animals being sedentary, they can witness the water quality evolution in a specific location.

The interest in investigating the bivalve’s activities by recording the valve movements has been explored for water quality surveillance. 
This area of interest is known as valvometry.
The basic idea of valvometry is to use the bivalve’s ability to close its shell when exposed to a contaminant as an alarm signal
(e.g.  \citealt{doherty1987valve}; \citealt{nagai2006detecting}; \citealt{MS11}). 
Thus, recording the shell gaping activity of oysters is an effective method to study their behavior when facing water pollution 
(e.g. \citealt{Riisgard2006}; \citealt{Garcia2008}). 
Nowadays, valvometric techniques produce high-frequency data, enabling online and {\it in situ} studies of the behavior of bivalve molluscs. 
They allow autonomous long-term recordings of valve movements without interfering their normal behavior.
The goal of this paper is to propose a nonparametric statistical procedure based on the estimation of the derivative of the
regression function in order to evaluate the velocity of the valve opening/closing activity. 
\vspace{1ex} \\
A wide range of literature is available on nonparametric estimation of a regression function. 
We refer the reader to  \citealt{NAD89},  \citealt{tsybakov2009introduction} and \citealt{devroye2012combinatorial}
for some excellent books on density and regression function estimation. 
Here, we shall focus our attention on the Nadaraya-Watson estimator of the regression function (\citealt{nadaraya1964estimating} and \citealt{watson1964smooth}). 
The almost sure convergence of this estimator was established by \citealt{noda1976estimation}, while its asymptotic normality  was proven by \citealt{schuster1972joint}. 
Later, \citealt{choi2000data} proposed three data-sharpening versions of the Nadaraya-Watson estimator in order to reduce the asymptotic variance in the central limit theorem.
\vspace{1ex} \\
In this paper, we investigate an alternative approach, based on three recursive versions of the Nadaraya-Watson estimator
(see \citealt{ahmad1976nonparametric}; \citealt{bercu2012robbins}; \citealt{devroye19801}; \citealt{johnston1982probabilities}; \citealt{wand1995kernel};
\citealt{duflo1997random}).
These recursive versions allows us to update the estimate with new collected information during the monitoring process. 
Consequently, it is possible to avoid the need to recompute a new final estimate from the whole data set. 
To the best of our knowledge, no references are available on the derivative of the recursive Nadaraya-Watson estimator. 
Our first goal is to study the asymptotic behavior of the derivative of those three estimators.
Our second goal is to illustrate our nonparametric estimation procedure on high-frequency valvometry data, 
in order to detect irregularities or abnormal behaviors of bivalves.
\vspace{1ex} \\
The paper is organized as follows. Section \ref{sec:NED} deals with our
nonparametric estimation procedure of the derivative of the regression function. We establish in Section \ref{sec:TR} the pointwise almost sure convergence as well as the asymptotic normality of our estimators and we compare their asymptotic variances. Section \ref{sec:SD} is devoted to simulation results to study the performance of our recursive procedure. Section \ref{sec:HFVD} presents an application for the survey of aquatic system using high-frequency valvometry. All the proofs of the nonparametric theoretical results are postponed to Appendices \ref{AppendixA} and \ref{AppendixB}.  

\section{Nonparametric estimation of the derivative}
\label{sec:NED}

The  relationship  between  the  distance  of two electrodes $(Y_{n})$ and the time of the measurement $(X_n)$ can be seen as 
a  nonparametric regression model given, for all $n\geq 1$, by
\begin{equation}
\label{MODEL}
Y_{n}=f(X_{n})+\veps_n
\end{equation}
where $(\veps_n)$ are unknown random errors. In all the sequel, we assume that $(X_n)$ is a sequence of independent and identically distributed
random variables with positive probability density function $g$.
Our purpose is to estimate the derivative of the unknown regression function $f$ which is directly associated with the velocity of the valve opening/closing activities of the oysters. For example, in an inhospitable environment, oysters behavior will be altered. Consequently, detecting changes of the closing and opening speed can provide insights about the health of oysters and so can be used as bioindicators of the water quality.
\vspace{2ex}\\
\noindent
We recall that the Nadaraya-Watson estimator of the link function $f$ is defined as 
\begin{equation}
\hat{f}^{NW}_n(x) = \dfrac{\skn Y_k K\left(\dfrac{x-X_k}{h_n}\right)}{\skn K\left(\dfrac{x-X_k}{h_n}\right) },
\label{DEFNW}
\end{equation}
where the kernel $K$ is a chosen probability density function and  
the bandwidth $(h_n)$ is a sequence of positive real numbers  
decreasing to zero. In our situation, we focus our attention on the recursive version of the
Nadaraya-Watson estimator (\citealt{duflo1997random}) of $f$ given, for any $x\in\dR$, by 
\begin{equation}
\fn = \dfrac{\skn \dfrac{Y_k}{h_k} \KXk}{\skn \dfrac{1}{h_k}\KXk}. 
\label{DEFFHAT}
\end{equation}
The denominator should, of course, be taken positive.
It coincides with the recursive version of the Parzen-Rosenblatt estimator
(\citealt{parzen1962estimation}; \citealt{rosenblatt1956remarks}) of the probability density function $g$. 
For any $x\in\dR$, denote
\begin{equation}
\hn = \dfrac{1}{n}\ \skn \dfrac{Y_k}{h_k}\KXk \hspace{0.8cm}\text{and}\hspace{0.8cm}
\gn = \dfrac{1}{n}\ \skn \dfrac{1}{h_k}\KXk,
\label{DEFHGHAT}
\end{equation} 
which can be recursively calculated as
\begin{equation}
\hn = \dfrac{n-1}{n}\ \hat h_{n-1}(x)+\dfrac{Y_n}{n h_n}\ K\left(\dfrac{x-X_n}{h_n}\right)
\label{DEFHGHAT_2}
\end{equation} 
and
\begin{equation}
\gn = \dfrac{n-1}{n}\ \hat g_{n-1}(x)+\dfrac{1}{n h_n}K\left(\dfrac{x-X_n}{h_n}\right).
\label{DEFHGHAT_3}
\end{equation} 
This modification allows dynamic updating of the estimates.
\vspace{2ex} \\
\noindent
In the special case where $g$ is known, a simplified version of the Nadaraya-Watson estimator
of $f$, introduced by \citealt{johnston1982probabilities}, is given by
\begin{equation}
\ftn = \dfrac{\hn}{g(x)}.
\label{DEFFT}
\end{equation}
In the same vein, an alternative estimator of $f$ when $g$ is known, was proposed by \citealt{wand1995kernel}. 
It is defined, for any $x\in\dR$, by
\begin{equation}
\fcn=\dfrac{1}{n}\skn \dfrac{Y_k}{g(X_k)h_k}\KXk.
\label{DEFFC}
\end{equation}
\\
\noindent
The derivatives of $\fn$, $\ftn$, and $\fcn$ are given, for any $x\in \dR$ such that $g(x)>0$, by
\begin{equation}
\fnp=\dfrac{\hnp}{\gn}-\dfrac{\hn \gnp}{\gnc}, 
\label{NWD}
\end{equation}
\begin{equation}
\ftnp=\dfrac{\hnp}{g(x)}-\dfrac{\hn g'(x)}{g^2(x)},
\label{JD}
\end{equation}
\begin{equation}
\fcnp=\dfrac{1}{n}\skn \dfrac{Y_k}{g(X_k)h_k^2}\KpXk.
\label{WJD}
\end{equation}

\section{Theoretical results}
\label{sec:TR}
In order to investigate the asymptotic behavior of these derivative estimates, it is necessary to introduce several
classical assumptions.
\vspace{1ex}
\begin{displaymath}
\begin{array}{ll}
(\mathcal{A}_1) & \textrm{The kernel $K$ is a positive symmetric bounded function, differentiable with}\\ 
  & \textrm{bounded derivative, satisfying}
\end{array}
\vspace{-1ex}
\end{displaymath} 
$$
\int_{\dR} K(x) dx =1, \hspace{0.4cm}\int_{\dR} K^\prime(x)dx=0,\hspace{0.4cm}
\int_{\dR} xK^\prime(x) dx =-1, \hspace{0.4cm} \int_{\dR} x^2 K^\prime(x) dx=0,
$$
$$
\int_{\dR} x^4 K(x) dx< \infty, \hspace{1cm}\int_{\dR} x^4 |K^\prime(x)| dx < \infty.
$$
\begin{displaymath}
\begin{array}{ll}
(\mathcal{A}_2) & \textrm{The regression function $f$ and the density function $g$ are bounded continuous,}\\
  & \textrm{twice differentiable with bounded derivatives.} \\
(\mathcal{A}_3) & \textrm{The noise sequence $(\veps_n)$ and the observation times $(X_n)$ are independent. }\\
  & \textrm{Moreover, $(\veps_n)$ is a sequence of independent, squared integrable, identically} \\
  & \textrm{distributed random variables such that $\dE[\veps_n]=0$ and $\dE[\veps_n^2]=\sigma^2$.}
 \end{array}
\end{displaymath}
Furthermore, the bandwidth $(h_n)$ is a sequence of positive real numbers, 
decreasing to zero, such that $n h_n$ tends to infinity. For the sake of simplicity, we shall make use of
$h_n = 1/n^{\alpha}$ with $0<\alpha <1$.
\vspace{2ex}\\
\noindent
Our first result on the almost sure convergence of our estimates is as follows.

\begin{thm}
\label{T-ASCVG}
Assume that $(\mathcal{A}_1)$, $(\mathcal{A}_2)$ and $(\mathcal{A}_3)$ hold. Then, if $0< \alpha <1/3$,
we have for any $x \in \dR$ such that $g(x)>0$, 
\begin{equation}
\label{ASCVGNWD}
\lim_{n\rightarrow \infty} \fnp = f^\prime(x) \hspace{1cm}\text{a.s.}
\end{equation}
\begin{equation}
\label{ASCVGJD}
\lim_{n\rightarrow \infty} \ftnp = f^\prime(x)
\hspace{1cm}\text{a.s.}
\end{equation}
\begin{equation}
\label{ASCVGWJD}
\lim_{n\rightarrow \infty} \fcnp = f^\prime(x)
\hspace{1cm}\text{a.s.}
\end{equation}
\end{thm}
\begin{proof}
The proof is given in Appendix \ref{AppendixA}.
\end{proof}
 \ \vspace{-2ex} \\ 
Our second result is devoted to the asymptotic normality of our estimates. Denote
\begin{equation}
\label{DefIntKp2}
\xi^2=\int_{\dR} \bigl(K^\prime(x)\bigr)^2 dx.
\end{equation}

\begin{thm}
\label{T-CLT}
Assume that $(\mathcal{A}_1)$, $(\mathcal{A}_2)$ and $(\mathcal{A}_3)$ hold and that the 
sequence $(\veps_{n})$ has a finite moment of order $>2$. Then, as soon as $1/5<\alpha<1/3$, 
we have for any $x \in \dR$ such that $g(x)>0$, the pointwise asymptotic normality
\begin{equation}
\label{CLTNWD}
\sqrt{nh_n^3}\bigl(\fnp-f^\prime(x)\bigr) \liml \cN \left (0,\dfrac{\xi^2}{(1+3\alpha)g(x)}\ \sigma^2\right ),
\end{equation}
\begin{equation}
\label{CLTJD}
\sqrt{nh_n^3}\bigl(\ftnp-f^\prime(x)\bigr) \liml \cN \left (0,\dfrac{\xi^2}{(1+3\alpha)g(x)}\bigl(f^2(x)+\sigma^2\bigr)\right ),
\end{equation}
\begin{equation}
\label{CLTWJD}
\sqrt{nh_n^3}\bigl(\fcnp-f^\prime(x)\bigr) \liml \cN \left(0,\dfrac{\xi^2}{(1+3\alpha)g(x)}\bigl(f^2(x)+\sigma^2\bigr)\right ).
\end{equation}
\end{thm}
\begin{proof}
The proof is given in Appendix \ref{AppendixB}.
\end{proof}

\begin{rem}
One can realize that the derivate of the Nadaraya-Watson estimator $\fnp$ is more efficient that $\ftnp$ and $\fcnp$
as its asymptotic variance is the smallest one. The more $f(x)$ is far away from $0$, the more one should make use of  $\fnp$.
\end{rem}

\section{Simulated data}
\label{sec:SD}
This section is devoted to numerical experiments in order to evaluate the performances of our derivative estimates. The data are generated from the nonparametric regression model
\begin{equation}
\label{modele}
Y_n=f(X_n)+\varepsilon_n,
\end{equation}
where the regression function $f$ is defined, for all $x$ in $[0,1]$, by 
\begin{equation}
\label{functionf}
f(x)=\sin(2\pi x^3)^3.
\end{equation}
The random observation $(X_n)$ is a sequence of independent random variables uniformly distributed over the interval $[0,1]$, and 
the source of variation $(\varepsilon_n)$ is a sequence of  independent and identically random variables sharing
the same $\mathcal N(0,1)$ distribution. We implement our statistical procedure with sample size $n=10\,000$. 
The simulated data associated with \eqref{modele} are given in Figure \ref{illustration}. 
\begin{figure}[htp]
\vspace{-3ex}
\begin{center}
\includegraphics[width=100mm,height=60mm]{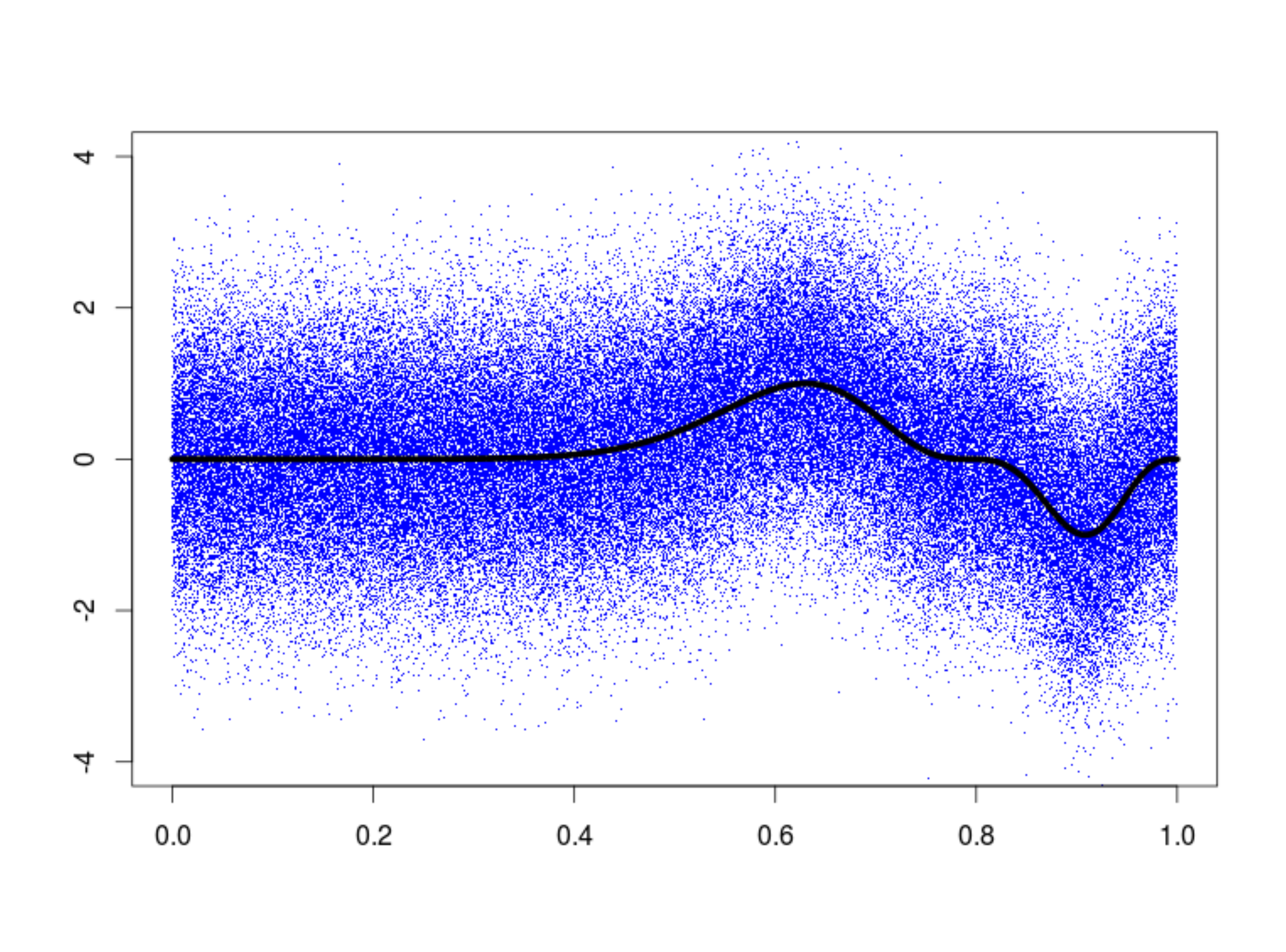} 
\vspace{-4ex}
\caption{Simulated data $(X_n,Y_n)$ with the regression function $f$ (solid line) given in (\ref{functionf}) with
sample size $n=10\, 000$.}
\label{illustration}
\end{center}
\end{figure}
\\
We first illustrate the pointwise almost sure convergence of the estimator $\widehat{f}^{\prime}_{n}$ to $f^\prime$ 
for the Gaussian and Epanechnikov kernels, respectively given by
$$
K(x)= \dfrac{1}{\sqrt{2\pi}} \exp \Bigl (-\frac{x^2}{2}\Bigr ) \qquad \mbox{and} \qquad K(x)=\dfrac{3}{4}\bigl(1-x^2\bigr)\rI_{|x|\leq 1}.
$$
The first kernel is supported on the whole real line, while the second one has a compact support.
The derivative $f^\prime$ is given, for all $x$ in $[0,1]$, by
\begin{equation}
\label{TrueDeriv}
f^\prime(x)=18 \pi \, x^2 \cos(2\pi x^3)\sin(2\pi x^3)^2.
\end{equation}
It is well-known that in practice, the choice of the kernel is not really significant, compared to the crucial choice of the bandwidth
$h_n=1/n^\alpha$. Figure \ref{Fig_compar_alpha} shows that one should take the value of $\alpha$ close to $0.3$.
\begin{figure}[tph]
\vspace{-3ex}
\begin{center}
\includegraphics[width=100mm,height=60mm]{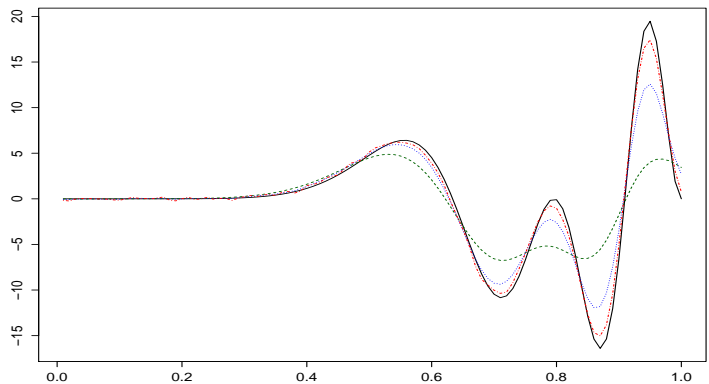}
\end{center}
\vspace{-3ex}
\caption{Representation of the estimator $\widehat{f}^{\prime}_{n}(x)$ of $f^{\prime}(x)$ given in (\ref{NWD}) with a Gaussian kernel for $\alpha=0.2$ (dashed line), $\alpha=0.25$ (dotted line) and $\alpha=0.3$ (dash-dotted line). The solid line is the underlying derivative $f^{\prime}(x)$.} 
\label{Fig_compar_alpha}
\end{figure}
In order to select an automatic choice of $\alpha$, we use the cross validation method by taking the value $\alpha$ 
that minimizes the cross validation function 
\begin{equation}
\label{CV}
CV(\alpha)=\dfrac{1}{n}\sum\limits_{k=1}^n \Bigl(\widehat{f}^\prime_{(-k)}(X_k,\alpha)  - f^\prime(X_k)\Bigr)^2
\end{equation}
where $\widehat{f}^\prime_{(-k)}(X_k,\alpha)$ is the estimator of $f^\prime(X_k)$ defined by (\ref{NWD}) with the couple $(X_k,Y_k)$ removed. Figure \ref{Fig_cross_val} displays the cross validation function for the Gaussian kernel. We observe between $\alpha=0.3$ and 
$\alpha=0.4$ a plateau leading to the value $\alpha=0.32$ since $\alpha$ has to be smaller than $1/3$. 
By the same method, we also obtain $\alpha= 0.32$ for the Epanechnikov kernel. 
\begin{figure}[tph]
\vspace{-3ex}
\begin{center}
\includegraphics[width=100mm,height=60mm]{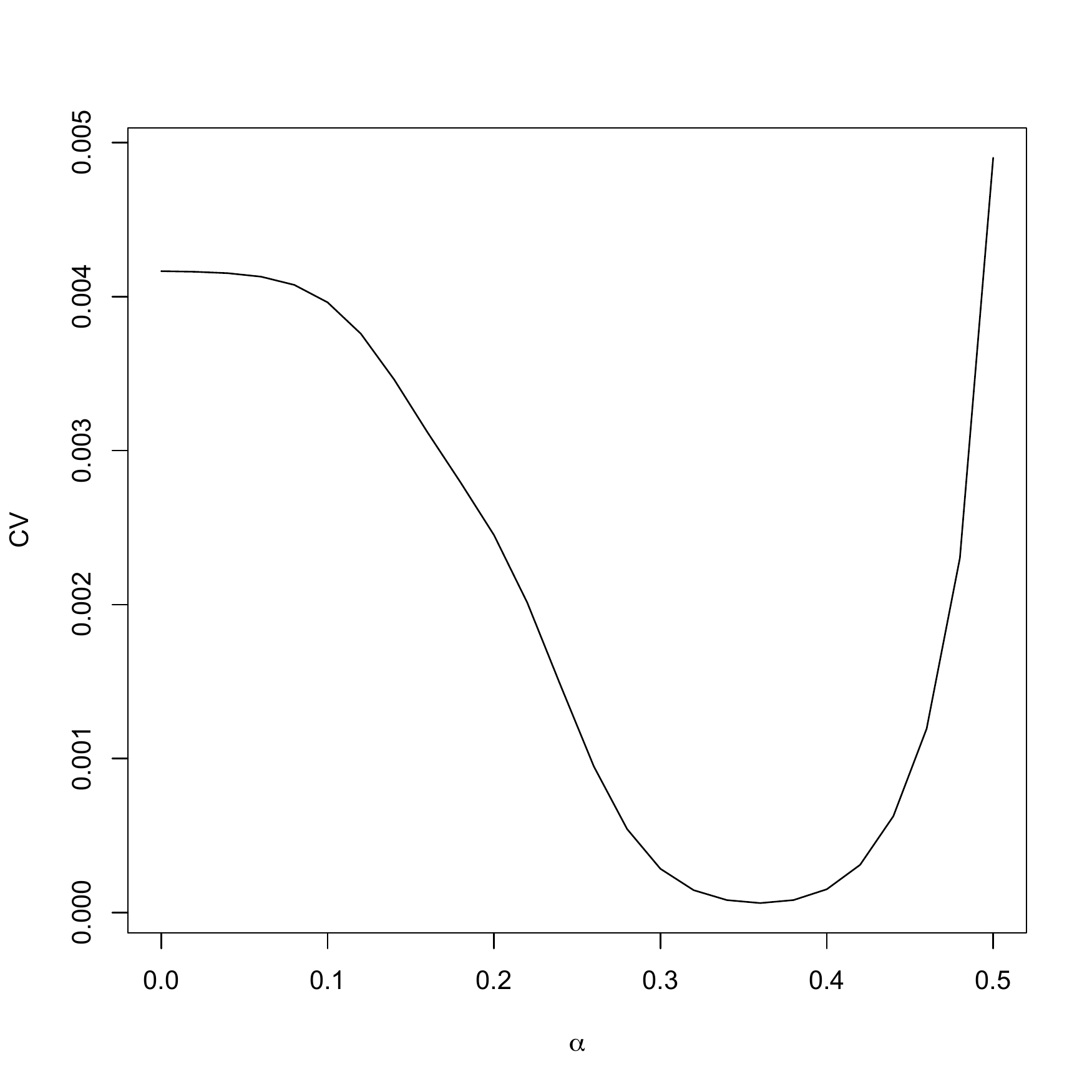}
\end{center}
\vspace{-3ex}
\caption{Example of the cross validation function for the estimates $\fnp$ of $f^{\prime}(x)$ with the Gaussian kernel.} 
\label{Fig_cross_val}
\end{figure}
\vspace{1ex}\\
Figure \ref{Fig_compar} illustrates the good approximation of $\widehat{f}_n^\prime$ to $f^{\prime}$
for the two kernels. Hereafter, we recall from Theorem \ref{T-CLT} that the asymptotic variance of our estimates depends on the integral
$\xi^2$ defined in \eqref{DefIntKp2}. Consequently, we select in our simulations the kernel with the smaller $\xi^2$ value.
It is easy to compute the values of  $\xi^2$ for the Gaussian and Epanechnikov kernels. They are respectively given 
by $\xi^2=1/4\sqrt{\pi} \simeq 0.1410$ and $\xi^2=3/2$. One can also observe that the Gaussian kernel has the smallest
$\xi^2$ value comparing to all commonly used kernels. Therefore, we shall use in all the sequel the Gaussian kernel.
\begin{figure}[tph]
\vspace{-2ex}
\begin{center}
\includegraphics[width=100mm,height=60mm]{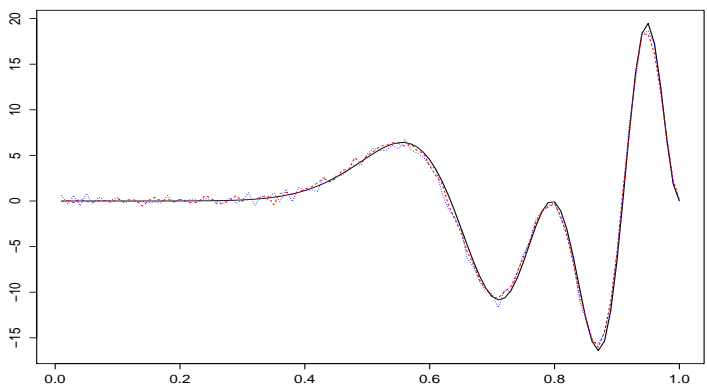}
\end{center}
\vspace{-2ex}
\caption{Representation of the estimator $\fnp$ of $f^{\prime}(x)$ for the Gaussian kernel (dashed line) and the Epanechnikov kernel (dotted line). The solid line is the underlying derivative $f^{\prime}(x)$.} 
\label{Fig_compar}
\end{figure}
\vspace{1ex}\\
After selecting $\alpha$ by cross validation, Figure \ref{Fig_Compare_Estim} shows that the three estimators $\fnp$, $\fcnp$, and $\ftnp$ approaches perfectly well the true derivative  $f^{\prime}(x)$.
 \begin{figure}[tph]
\vspace{-3ex}
\begin{center}
\includegraphics[width=100mm,height=60mm]{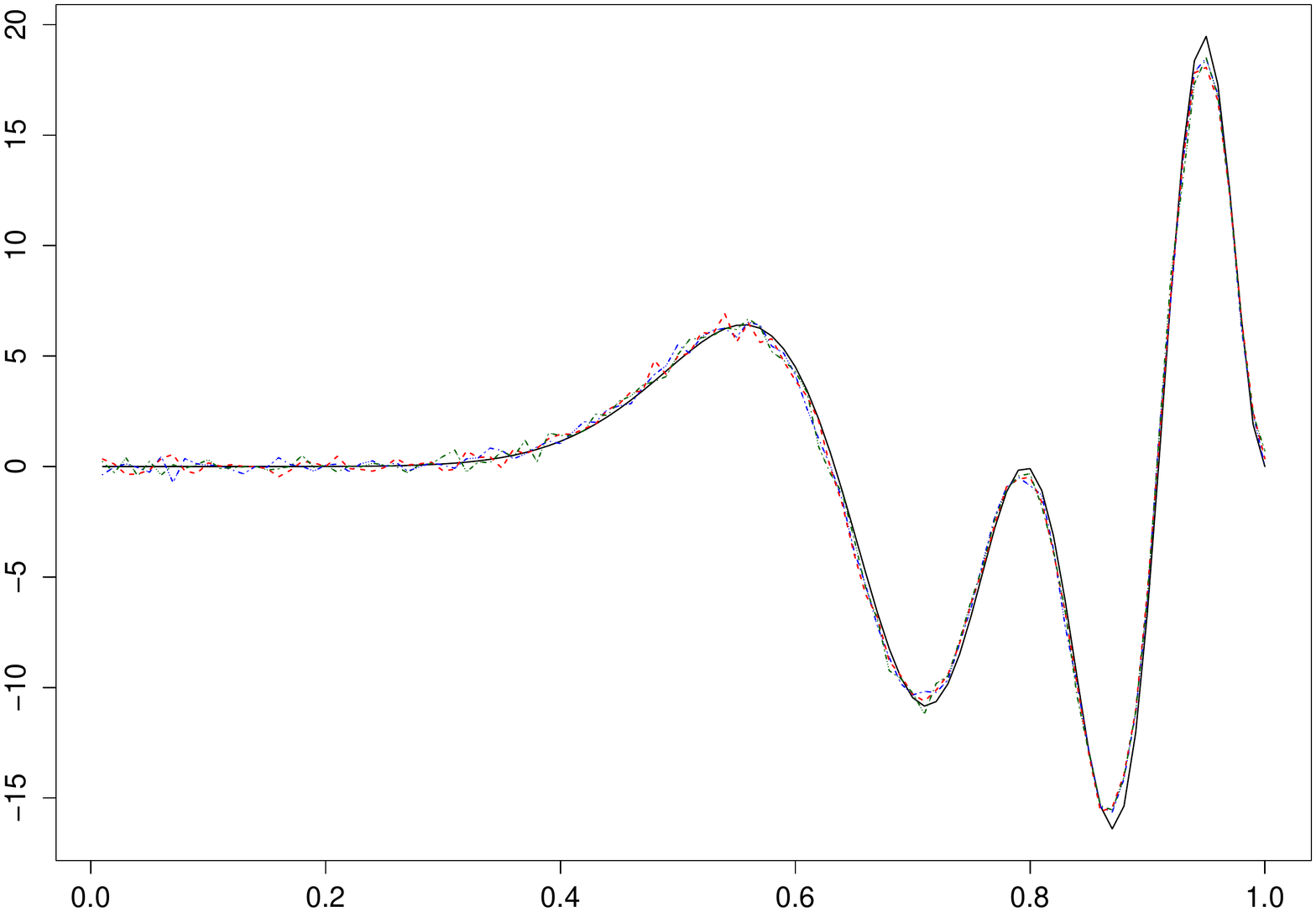}
\end{center}
\vspace{-3ex}
\caption{Illustration of the almost sure convergence of $\fnp$ (dotted line), $\ftnp$ (dashed line), and $\fcnp$ (dash-dotted line), to $f^{\prime}(x)$ (solid line).} 
\label{Fig_Compare_Estim}
\end{figure}
\vspace{1ex}\\
In order to illustrate the pointwise asymptotic normality of our estimates,  we implement a simulation study based on
$N=2\,000$ realizations.
We numerically check  the asymptotic normality at points $x=0.4$ and $x=0.9$ for our three estimators. 
One can see in Figure \ref{TLC_x} that the distributions of our three estimators are normally distributed and centered around $0$. We observe the effect of $f^{2}(x)$ in the asymptotic variance of $\ftnp$ and $\fcnp$. Indeed, for $x=0.4$ we have $f^2(x)=0.0036$, while for $x=0.9$ we have $f^2(x)=0.9489$ which explains the differences between the asymptotic variances.
\begin{figure}[htbp]
    \centering
%
%
\includegraphics[scale=0.4]{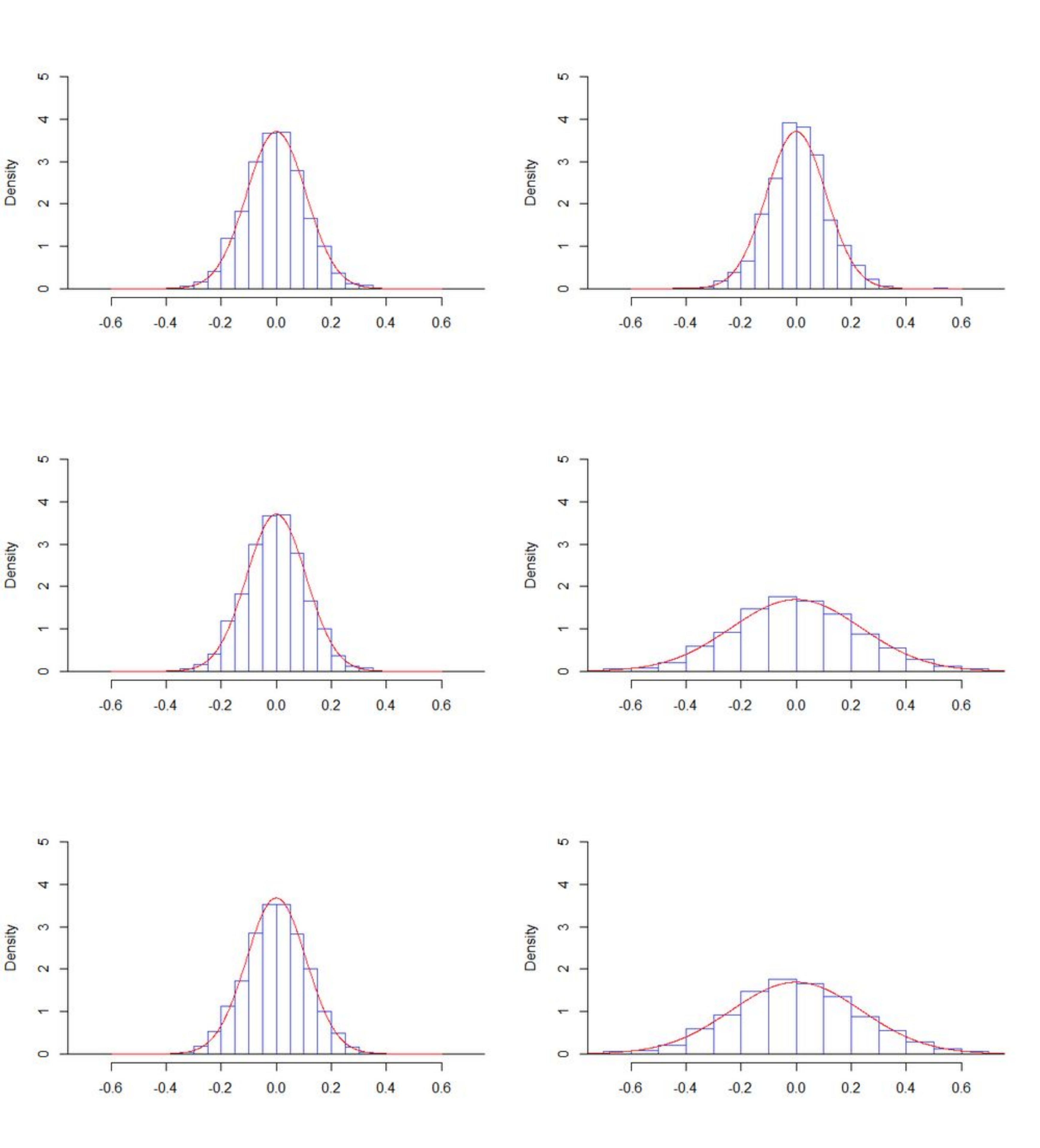}
\caption{Asymptotic normality of $\fnp$ (first row), $\ftnp$ (second row) and $\fcnp$ (third row) at point $x=0.4$ (left column) and 
point$x=0.9$ (right column). The density curves represent the asymptotic normal distributions given in Theorem \ref{T-CLT}.}
\label{TLC_x}
\end{figure}
\vspace{1ex}\\
One can observe in Figure \ref{Boxplot} the different behavior of the asymptotic variance
$\widehat f^\prime_n(x)$ in comparison with the two others estimators. Once again, a high variability coincides with a 
large value of $f^2(x)$.
\begin{figure}[htbp]
    \centering
\includegraphics[scale=0.75]{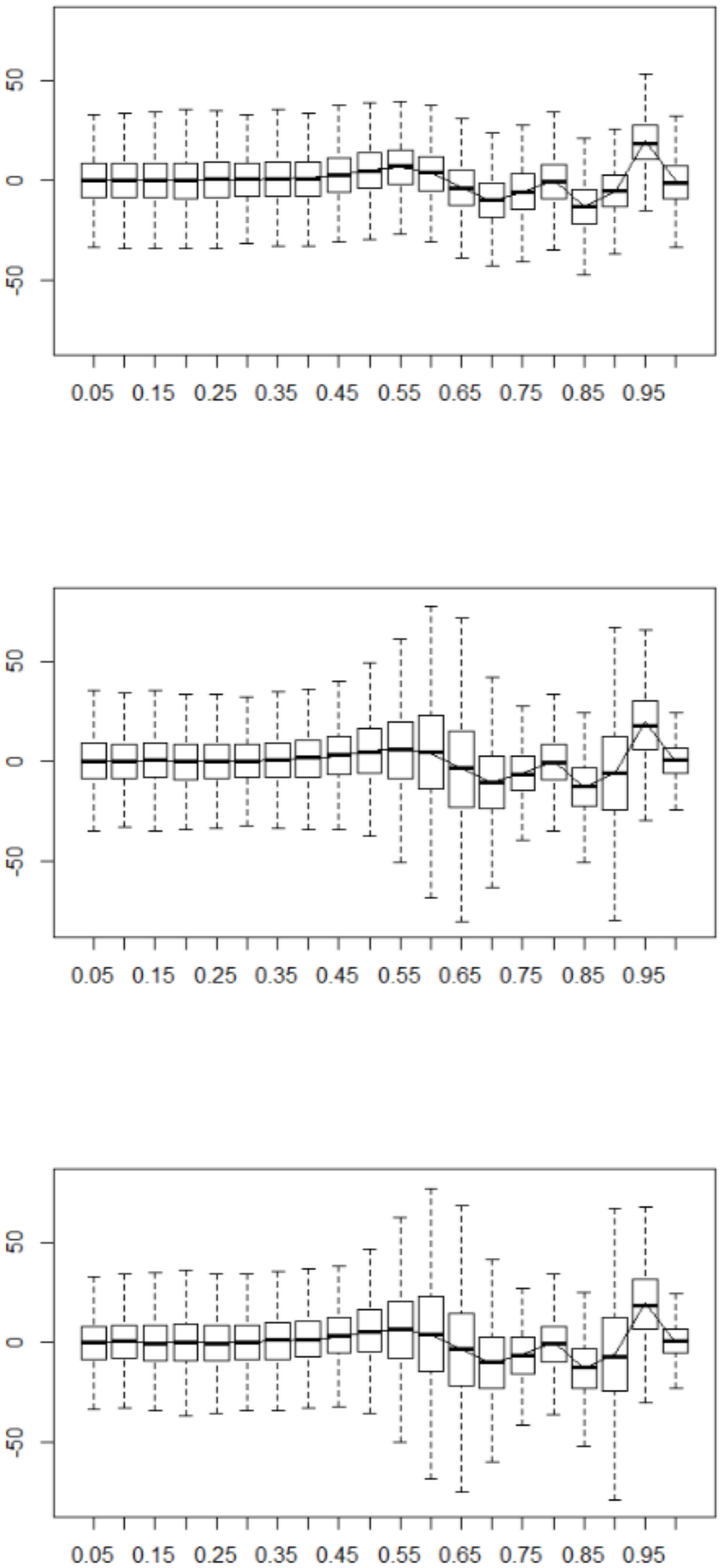}
\caption{Boxplots of the three estimates $\fnp$, $\fcnp$ and $\ftnp$ of $f^{\prime}(x)$ versus $x$. The solid line represents the underlying derivative function $f^{\prime}(x)$.}
\label{Boxplot}
\end{figure}
\vspace{1ex}\\
Finally, our numerical experiments illustrate the good performances and also the robustness of our statistical procedure for heavy-tailed error distributions \citealt{durrieu2009sequential} (data not shown). We also observed that the mean squared error of $\widehat f_n^\prime$ is much more smaller than
the mean squared error of the non-recursive version of the Nadaraya-Watson estimator. In term of asymptotic variance, 
it is clear that $\widehat f_n^\prime$ performs better than $\widetilde f_n^\prime$ and $\widecheck f_n^\prime$. Consequently,
we choose to make use of $\widehat f_n^\prime$ to estimate the derivative $f^{\prime}$ for our real life data
experiments.




\section{High-frequency valvometry data}
\label{sec:HFVD}

The motivation of this paper is to monitor sea shores water quality. For that purpose, we study bivalves activities by recording the valve movements. We use a high frequency, noninvasive valvometry electronic system developed by the UMR CNRS 5805 EPOC laboratory in Arcachon (France). The electronic principle of valvometry is described by \citealt{Tran2003}, \citealt{Chambon2007} and on the
website \url{http://molluscan-eye.epoc.u-bordeaux1.fr}. This electronic system works autonomously without human intervention for a long period of time (at least one full year). Each animal is equipped with two light coils (sensors), of approximately $53$mg each (unembedded), fixed on the edge of each valve. One of the coils emits a high-frequency, sinusoidal signal which is received by the other coil. The strength of the electric field produced between the two coils being proportional to the inverse of distance between the point of measurement and the center of the transmitting coil, the distance between coils can be measured and the accuracy of the measurements is a few $\mu$m.
\vspace{1ex}\\
For each sixteen animals, one measurement is received every $0.1$s (10 Hz). So, the activity of oyster is measured every $1.6$s and 
each day, we obtain $864\,000$ triplets of data: the time of the measurement, the distance between the two valves and the animal number. A first electronic card in a waterproof case next to the animals manages the electrodes and a second electronic card handles the data acquisition. The  valvometry system uses a GSM/GPRS modem and Linux operating system for the data storage, the internet access, and the data transmission. After each $24$h period or any other programmed period of time, the data are transmitted to a workstation server and then inserted in a SQL database which is accessible with the software R (\citealt{Rsoft}) or a text terminal.  
\vspace{1ex}\\
Several valvometric systems have been installed around the world: southern lagoon of New Caledonia, Spain, Ny Alesund Svalbard at $1300$ km from the north pole, the north east of Murmansk in Russia on the Barents sea and at several sites in France with various species but we concentrate here on  the Locmariaquer site situated in south Brittany based on sixteen oysters placed in a single bag. Locmariaquer (GPS coordinates $47^{\circ}34$ N, $2^{\circ}56$ W) is an important oyster farming area located near the narrow tidal pass which connects the gulf of Morbihan to the ocean, on the right side of the Auray river's mouth. Thus, oysters are close to the seasonal high traffic of the navigation channel and are potentially exposed to pollution as chemical residues of intensive agricultural practices.  
\vspace{1ex}\\
As argued in \citealt{ahmed2015}, \citealt{DurrieuAl2015} and \citealt{DurrieuAl2016},  pollution can affect the activity of oysters and in particular the shells opening and closing velocities and so the movement speeds can be considered as an indicator of the animal stress activity since its movements are associated to aquatic system perturbations. In \citealt{ahmed2015}, the authors propose an interesting deterministic alternative method for the estimation of movement velocity based on differentiator estimators. \\
An example  of valves activity and opening/closing velocity recordings June $2$, $2011$ is depicted in Figure \ref{Valvometric_signal}. Figure \ref{Valvometric_Estimator1} displays for the same day the plot of the estimate $\widehat f_n^\prime$ of $f^{\prime}$ of the valve closing and opening velocity for one oyster at the Locmariaquer site.  The bandwidth parameter was selected by the cross validation method
described in the previous section.
\vspace{1ex}
\begin{figure}[tph]
\begin{center}
\includegraphics[scale=0.4]{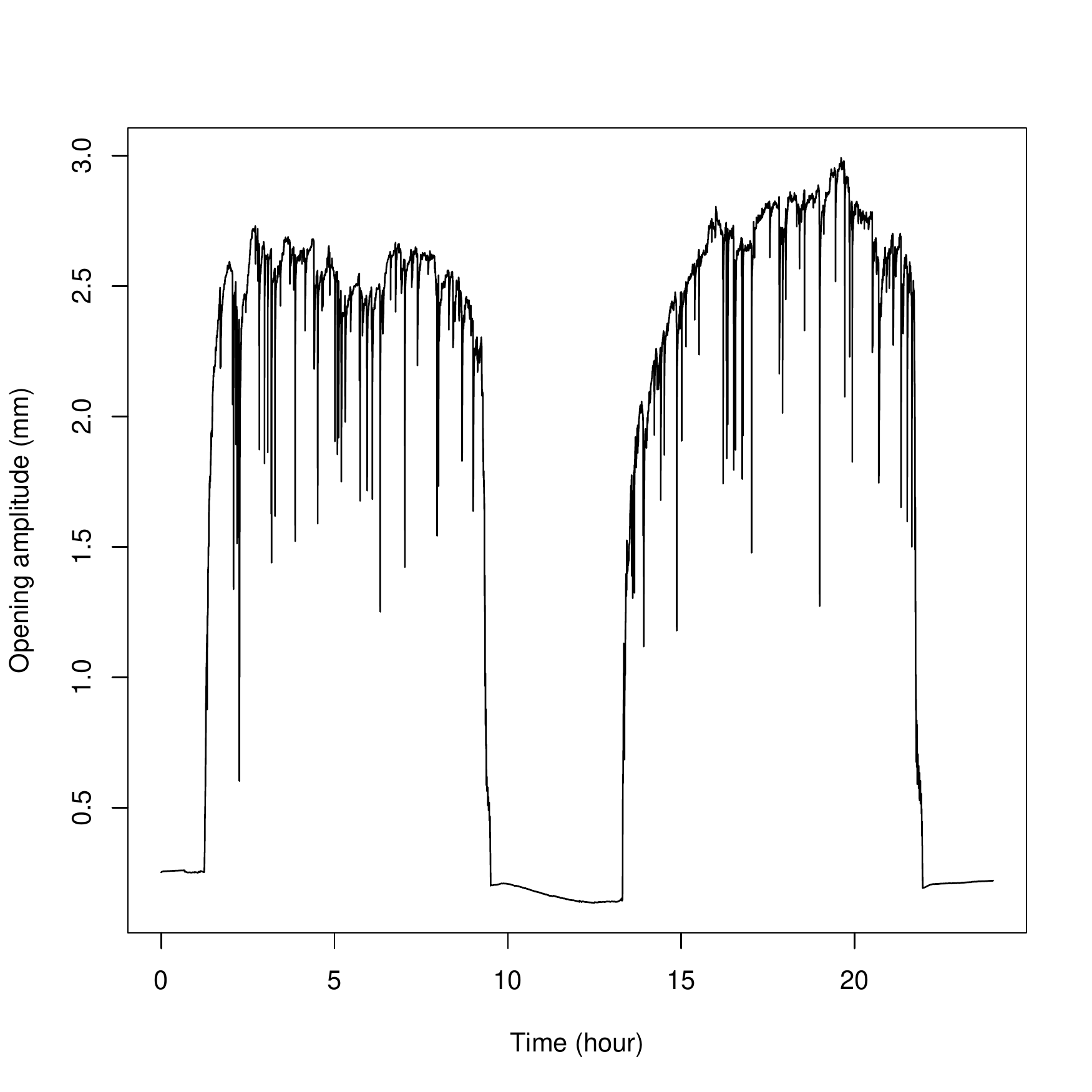}
\includegraphics[scale=0.4]{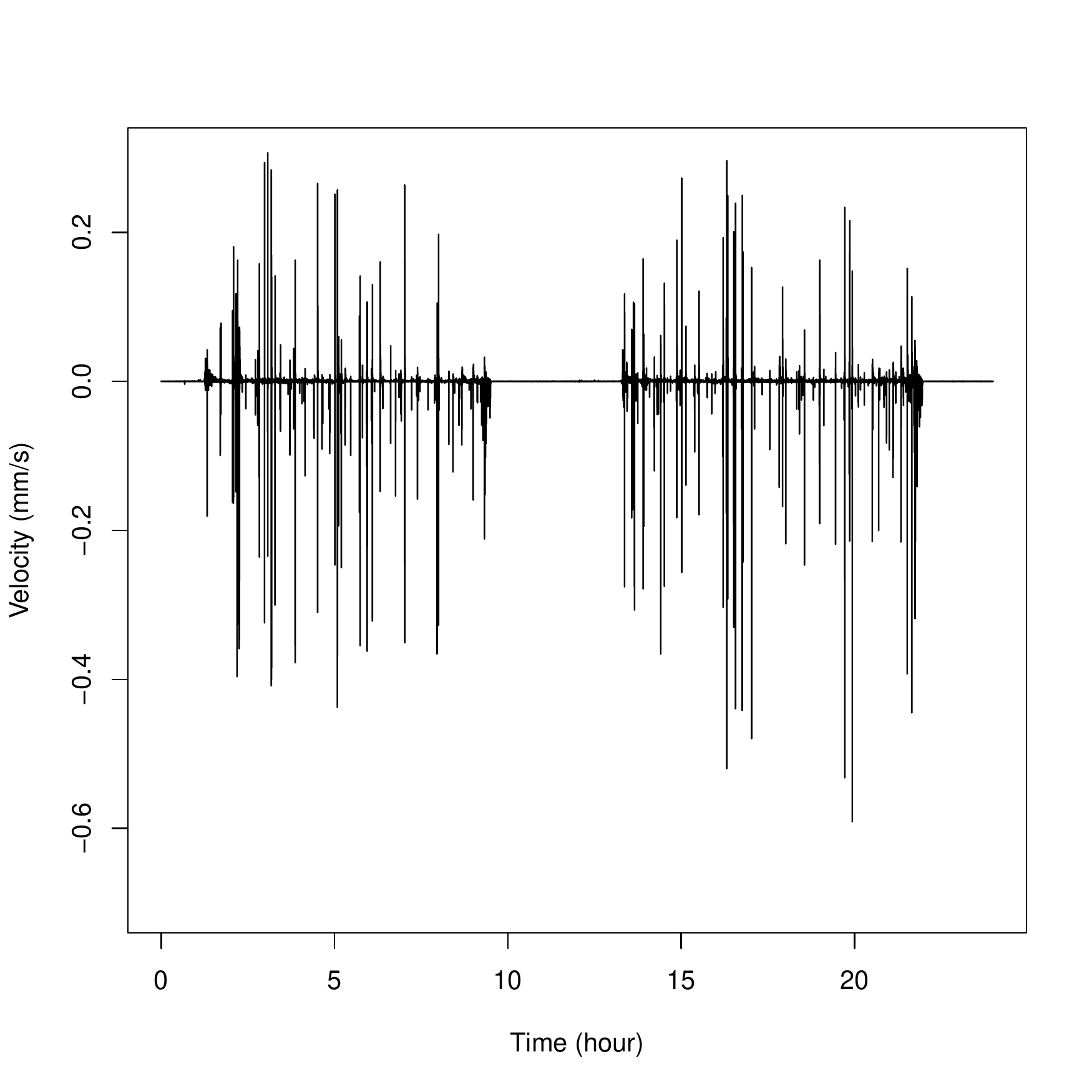}
\caption{A typical example of valvometric data for one oyster the June $2$, $2011$. In the left hand side, relationship between the opening amplitude (in millimeters) and the time of the experiment (over $24$ hours period). In the right hand side, the closing and opening velocity (millimeters per second) according to time (over the same period).}
\label{Valvometric_signal}
\end{center}
\end{figure}
\begin{figure}[tph]
\begin{center}
\includegraphics[width=120mm,height=70mm]{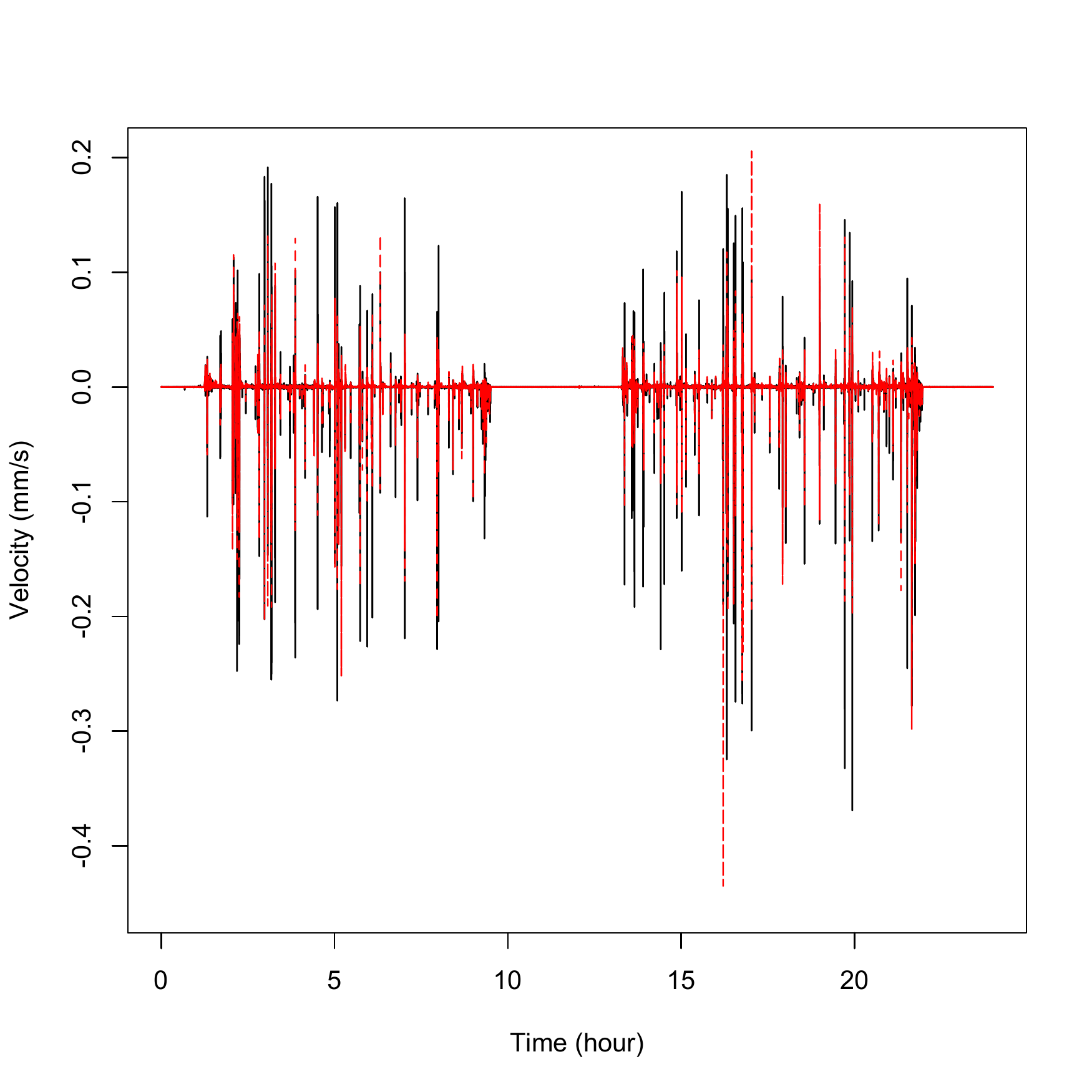}
\caption{The dashed line displays for June $2$, $2011$, the estimated $f^{\prime}(x)$ using estimator $\widehat f_n^\prime(x)$ versus the time $x$ and the solid lines represent the observed speeds of valve openings and closings. The closing and opening velocity are measured in millimeters per second.}
\label{Valvometric_Estimator1}
\end{center}
\end{figure}

\noindent To visualize the opening and closing velocity estimations of the $16$ oysters from the $63$th to the $243$th days of $2011$, we represent in Figure \ref{Fig-Vitesse-All}  for each oyster and each day the estimator $\widehat f_n^\prime(x)$ of the closing and opening velocities $f^{\prime}(x)$  at time $x$ over a period of time of 24 hours using a customized color table: the yellow color is associated to the class of the smallest velocities, the green color to the class of intermediate velocities and the red color to the class of the largest velocities. This graphical representation reveals distinct clusters of behaviors. We observe in the Figure \ref{Fig-Vitesse-All} white lines from the $207$th to the $210$th days, corresponding to a power outage on the site due to a storm. Before the $100$th days, the animals have a normal regular activity. The most red zone between the $100$th days and the $125$th days can be explained by a sudden change in temperature in the environment associated to the modifications of the specific activity of two enzymatic biomarkers (Glutathione-S-transferase and Acetylcholinesterase) meaning a possible pollution as described in  \citealt{DurrieuAl2016}. We observed an intense activity of closing at the bottom of the Figure (days $\geq 210$) associated to a spawning activity. 
\vspace{1ex}\\
Figure \ref{Fig-Vitesse-All} shows also that the closing and opening velocities are the smallest (yellow zone) and highly correlated with the tidal amplitude. We have performed many other analyses of these data using extreme value theory and other nonparametric statistical methods, all of which point the same conclusion \citealt{CoudretAl2015}, \citealt{DurrieuAl2015} and \citealt{DurrieuAl2016}. Altogether, we anticipate that this approach could have a significant contribution providing {\it in situ} instant diagnosis of the bivalves behavior and thus appears to be an effective, early warning tool in ecological risk assessment.


\begin{figure}[tph]
    \centering
\includegraphics[width=160mm,height=130mm]{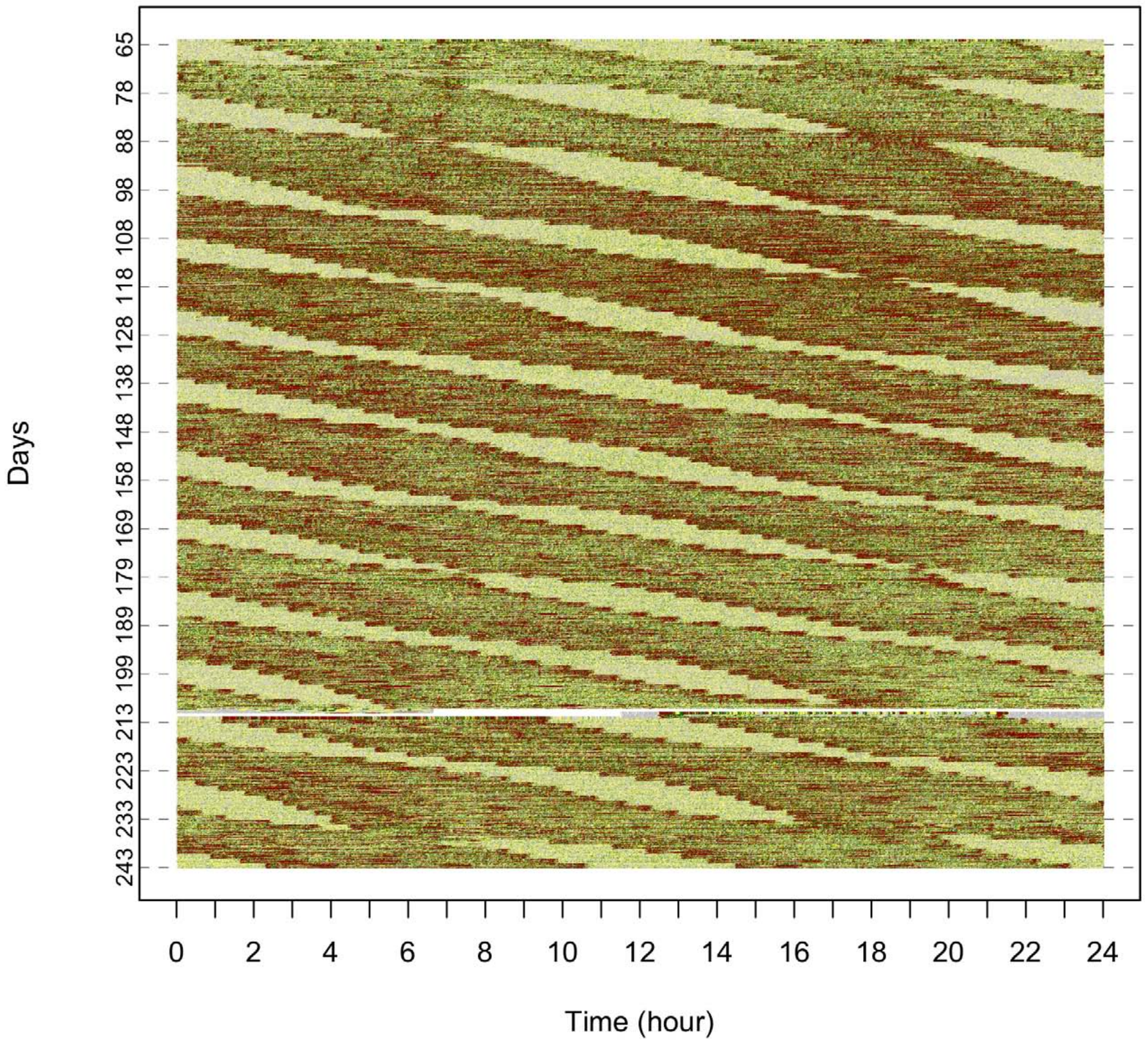}
\caption{Representation of the opening and closing velocities estimation using $\widehat f_n^\prime(x)$ from the
$63$th to the $243$th days of $2011$, considering the $16$ oysters in Locmariaquer. The x-axis represents the time in a $24$ hour time period and the y-axis represents the number of days since January $1$, $2011$.}
\label{Fig-Vitesse-All}
\end{figure}

\newpage
\appendix
\section{Proofs of the almost sure convergence results.}
\label{AppendixA}
\renewcommand{\thesection}{\Alph{section}}
\renewcommand{\theequation}{\thesection.\arabic{equation}}
\setcounter{section}{1}
\setcounter{equation}{0}
%

The proofs of the almost sure convergence results rely on the following lemma. 
We also refer the reader to \citealt{silverman1986density} for the estimation of the 
derivative of the Parzen-Rosenblatt estimator.

\begin{lem}
\label{L-ASCVG}
Assume that $(\mathcal{A}_1)$, $(\mathcal{A}_2)$ and $(\mathcal{A}_3)$ hold. Then, 
the estimators $\wh{g}_n$ and $\wh{h}_n$, given by \eqref{DEFHGHAT}, satisfy
for any $x \in \dR$, 

\begin{equation}
\label{ASCVGG}
\lim_{n\rightarrow \infty} \gn = g(x) \hspace{1cm}\text{a.s.}
\end{equation} 

\begin{equation}
\label{ASCVGH}
\lim_{n\rightarrow \infty} \hn = f(x)g(x) \hspace{1cm}\text{a.s.}
\end{equation} 

\noindent
Moreover, as soon as $0<\alpha<1/3$, we also have for any $x \in \dR$,

\begin{equation}
\label{ASCVGGP}
\lim_{n\rightarrow \infty} \gnp = g^\prime(x) \hspace{1cm}\text{a.s.}
\end{equation} 

\begin{equation}
\label{ASCVGHP}
\lim_{n\rightarrow \infty} \hnp = \bigl(f(x)g(x)\bigr)^\prime \hspace{1cm}\text{a.s.}
\end{equation} 

\end{lem}

\begin{proof}
We shall only prove the almost sure convergence \eqref{ASCVGHP} inasmuch as 
\eqref{ASCVGG} and \eqref{ASCVGH} are well-known and
the proof of \eqref{ASCVGGP} is more easy to handle and follow the same lines
as the proof of \eqref{ASCVGHP}. 
We deduce from \eqref{MODEL} and \eqref{DEFHGHAT} that for any $x \in \dR$,
$$
\hn =\dfrac{1}{n} \skn \dfrac{f(X_k)}{h_k} \KXk + \dfrac{1}{n} \skn \dfrac{\veps_k}{h_k} \KXk.
$$
Hence, by derivation, we have the decomposition
\begin{equation}
\label{DECOHN}
 n\hnp = A_n(x) + B_n(x)
\end{equation}
where 
\begin{align*}
A_n(x) &= \skn a_k(x)= \skn f(X_k)v_k(X_k,x), \\
B_n(x) &= \skn b_k(x)= \skn \veps_k v_k(X_k,x) 
\end{align*}
with
\begin{equation}
\label{vn}
v_n(X_n,x)= \dfrac{1}{h_n^2} \KpXn.
\end{equation}
On the one hand, we have for any $x \in \dR$,
\begin{align}
\dE [a_n(x)] &= \int_{\dR} f(x_n) v_n(x_n,x) g(x_n) dx_n \notag\\
&=\dfrac{1}{h_n}\int_{\dR} f(x-h_ny)g(x-h_ny)K^\prime(y) dy.
\label{PRODAN}
\end{align}
The regression function $f$ as well as the density function $g$ are bounded continuous
and twice differentiable with bounded derivatives. Consequently, it follows from
Taylor's formula that it exist $\theta_f, \theta_g$ in the interval $]0,1[$ such that, for any $x \in \dR$,
$$
f(x-h_ny) = f(x)-h_n y f^\prime(x)+\dfrac{h_n^2y^2}{2}f^{\prime \prime}(x-h_ny \theta_f), 
$$
and 
$$
g(x-h_ny) = g(x)-h_n y g^\prime(x)+\dfrac{h_n^2y^2}{2}g^{\prime \prime}(x-h_ny \theta_g).
$$
By a careful analysis of each term in the product $f(x-h_ny)g(x-h_ny)$, we deduce from 
\eqref{PRODAN} together with assumption $(\mathcal{A}_1)$ that
\begin{align}
\dE [a_n(x)] & =  - \bigl(f(x)g(x)\bigr)^\prime \int_{\dR} yK^\prime(y) dy
+h_n f^\prime(x)g^\prime(x) \int_{\dR} y^2K^\prime(y) dy + R_n(x) \notag \\
& =  \bigl(f(x)g(x)\bigr)^\prime + R_n(x)
\label{MEANAN}
\end{align}
where the remainder $R_n(x)$ satisfies
$$
\sup_{x \in \dR} |R_n(x)|=O (h_n).
$$
Consequently, \eqref{MEANAN} immediately leads to
\begin{equation}
\lim_{n\rightarrow \infty} \frac{1}{n}\dE[A_n(x)]= \bigl(f(x)g(x)\bigr)^\prime,
\label{CVGMEANAN}
\end{equation}
which is the limit we are looking for. By the same token,
\begin{align}
\dE [a_n^2(x)] & = \int_{\dR} f^2(x_n) v_n^2(x_n,x) g(x_n) dx_n \notag\\
& =\frac{1}{h_n^3}\int_{\dR} f^2(x-h_ny)g(x-h_ny)\bigl(K^\prime(y)\bigr)^2 dy \notag \\
& = \frac{1}{h_n^3} \xi^2 f^2(x)g(x) + \zeta_n(x),
\label{M2AN}
\end{align}
where $\xi^2$ is defined in \eqref{DefIntKp2} and 
the remainder $\zeta_n(x)$ is such that
$$
\sup_{x \in \dR} |\zeta_n(x)|=O \Bigl(\frac{1}{h_n^2}\Bigr).
$$
Therefore, we deduce from \eqref{MEANAN} and \eqref{M2AN} that
\begin{equation}
\lim_{n\rightarrow \infty} \frac{1}{n^{1+3\alpha}}\text{Var}(A_n(x))= \frac{\xi^2 f^2(x)g(x)}{1+3\alpha}.
\label{CVGVARAN}
\end{equation}
On the other hand, denote by $\cF_n$ the $\sigma$-algebra of the events occurring up to time $n$,
$\cF_n=\sigma(X_1,\veps_1,\ldots,X_n,\veps_n)$. Since $(X_n)$ and $(\veps_n)$ are two independent
sequences of independent and identically distributed random variables, we have for any $x \in \dR$,
$$\dE[b_n(x)|\cF_{n-1}]= \dE[\veps_nv_n(X_n,x) |\cF_{n-1}]=\dE[\veps_nv_n(X_n,x)]=0.$$
Moreover,
$$
\dE[b_n^2(x)|\cF_{n-1}]= \dE[\veps_n^2v_n^2(X_n,x) |\cF_{n-1}]=\dE[\veps_n^2v_n^2(X_n,x)]=\sigma^2\dE[v_n^2(X_n,x)].
$$
Furthermore, we have
\begin{align}
\dE [v_n^2(X_n, x)] & = \int_{\dR}  v_n^2(x_n,x) g(x_n) dx_n =\frac{1}{h_n^3}\int_{\dR} g(x-h_ny)\bigl(K^\prime(y)\bigr)^2 dy \notag \\
& = \frac{1}{h_n^3} \int_{\dR} \Bigl(g(x)-h_n y g^\prime(x)+
\dfrac{h_n^2y^2}{2}g^{\prime \prime}(x-h_ny \theta_g)\Bigr)\bigl(K^\prime(y)\bigr)^2 dy \notag \\
& = \frac{1}{h_n^3} \xi^2 g(x)  + \Delta_n(x)
\label{M2VN}
\end{align}
where $\xi^2$ is defined in \eqref{DefIntKp2}
and the remainder $\Delta_n(x)$ is such that
$$
\sup_{x \in \dR} |\Delta_n(x)|=O \Bigl(\frac{1}{h_n^2}\Bigr).
$$
Consequently, denoting
$$
W_n(x)=\skn v_k^2(X_k,x),
$$
it follows from \eqref{M2VN} that
\begin{equation}
\lim_{n\rightarrow \infty} \frac{1}{n^{1+3\alpha}}\dE[W_n(x)]
= \frac{\xi^2 g(x)}{1+3\alpha}.
\label{CVGMEANVN}
\end{equation}
We are now in the position to prove the almost sure convergence \eqref{ASCVGHP}. The decomposition 
\eqref{DECOHN} can be rewritten as
\begin{equation}
\label{NEWDECOHN}
n \hnp = M_n^{\!A}(x) + \dE[A_n(x)] + B_n(x),
\end{equation}
where $M_n^{\!A}(x)=A_n(x)-\dE[A_n(x)]$. One can observe that $(M_n^{\!A}(x))$ and $(B_n(x))$ are both
square integrable martingale difference sequences with predictable quadratic variations respectively given by
$\langle M^{\!A}(x) \rangle_n\!=\text{Var}(A_n(x))$ and $\langle B(x) \rangle_n\!=\sigma^2\dE[W_n(x)]$. Consequently,
\eqref{CVGVARAN} together with \eqref{CVGMEANVN} immediately lead to
\begin{equation}
\label{CVGQVMNBN}
\lim_{n\rightarrow \infty} \frac{\langle M^A(x) \rangle_n}{n^{1+3\alpha}}=\frac{\xi^2 f^2(x)g(x)}{1+3\alpha}
\hspace{1cm}\text{and}\hspace{1cm}
\lim_{n\rightarrow \infty} \frac{\langle B(x) \rangle_n}{n^{1+3\alpha}}=\frac{\sigma^2 \xi^2 g(x)}{1+3\alpha}.
\end{equation}
Hence, we obtain from the strong law of large numbers for martingales
given e.g. by Theorem 1.3.15 of \cite{duflo1997random} that, for any $\gamma >0$,
$(M_n^{\!A}(x))^2=o(n^{1+3\alpha}(\log n)^{1+\gamma})$ a.s. and $(B_n(x))^2=o(n^{1+3\alpha}(\log n)^{1+\gamma})$ a.s.  
Therefore, as $0<\alpha<1/3$, it ensures that, for any $x\in \dR$
\begin{equation}
\label{ASCVGMNBN}
\lim_{n\rightarrow \infty} \frac{1}{n} M_n^{\!A}(x) = 0 \hspace{0.3cm} \text{a.s.}
\hspace{1cm}\text{and}\hspace{1cm}
\lim_{n\rightarrow \infty} \frac{1}{n} B_n(x) = 0 \hspace{0.3cm} \text{a.s.}
\end{equation}
Finally, we deduce from decomposition \eqref{NEWDECOHN} together with \eqref{CVGMEANAN} and \eqref{ASCVGMNBN} that
for any $x \in \dR$,
\begin{equation*}
\lim_{n\rightarrow \infty} \hnp = \bigl(f(x)g(x)\bigr)^\prime \hspace{1cm}\text{a.s.}
\end{equation*} 
Thus Lemma \ref{L-ASCVG} is proven.
\end{proof}

\noindent
{\bf Proof of Theorem \ref{T-ASCVG}.} We shall now proceed to the proof of the Theorem \ref{T-ASCVG}.
It clearly follows from relation \eqref{NWD} and Lemma \ref{L-ASCVG} that for any $x \in \dR$ such that
$g(x)>0$,
\begin{align*}
 \lim_{n\rightarrow +\infty}\fnp & =\lim_{n\rightarrow +\infty}\Bigl(\frac{\hnp}{\gn}-\frac{\hn \gnp}{\gnc}\Bigr)
 = \frac{\bigl(f(x)g(x)\bigr)^\prime}{g(x)}-\frac{f(x)g(x) g^{\prime}(x)}{g^2(x)} \hspace{1cm}\text{a.s.}\\ 
 & = \frac{f^\prime(x)g(x)+f(x)g^{\prime}(x)-f(x)g^{\prime}(x)}{g(x)}=f^\prime(x)\hspace{1cm}\text{a.s.} 
 \end{align*}
By the same token, relation \eqref{JD} and Lemma \ref{L-ASCVG} immediately lead to
$$
\lim_{n\rightarrow +\infty} \ftnp  =\lim_{n\rightarrow +\infty}\Bigl(\frac{\hnp}{g(x)}-\frac{\hn g'(x)}{g(x)^2}\Bigr)=f^\prime(x)\hspace{1cm}\text{a.s.} 
$$
It only remains to prove \eqref{ASCVGWJD}. We obtain from relation \eqref{WJD} that
\begin{equation}
\label{DECOFCPN}
 n\fcnp = C_n(x) + D_n(x)
\end{equation}
where 
\begin{align*}
C_n(x) &= \skn c_k(x)= \skn \frac{f(X_k)}{g(X_k)}v_k(X_k,x), \\
D_n(x) &= \skn d_k(x)= \skn \frac{\veps_k}{g(X_k)}v_k(X_k,x). 
\end{align*}
As in the proof of Lemma \ref{L-ASCVG}, we find that for any $x \in \dR$ such that
$g(x)>0$,
\begin{equation}
\lim_{n\rightarrow \infty} \frac{1}{n}\dE[C_n(x)]= f^\prime(x)
\label{CVGMEANVARCN}
\hspace{0.5cm}
\text{and}
\hspace{0.5cm}
\lim_{n\rightarrow \infty} \frac{1}{n^{1+3\alpha}}\text{Var}(C_n(x))= \frac{\xi^2 f^2(x)}{(1+3\alpha)g(x)}.
\end{equation}
Hereafter, we split $n\fcnp$ into three terms
\begin{equation}
\label{NEWDECOFCPN}
 n\fcnp = M_n^{C}(x) + \dE[C_n(x)] + D_n(x)
\end{equation}
where $M_n^{C}(x)=C_n(x)-\dE[C_n(x)]$. One can observe that $(M_n^{C}(x))$ and $(D_n(x))$ are both
square integrable martingale difference sequences with predictable quadratic variations satisfying,
for any $x \in \dR$ such that $g(x)>0$,
\begin{equation}
\label{CVGQVNNDN}
\lim_{n\rightarrow \infty} \frac{\langle M^{C}(x) \rangle_n}{n^{1+3\alpha}}=\frac{\xi^2 f^2(x)}{(1+3\alpha)g(x)}
\hspace{0.5cm}\text{and}\hspace{0.5cm}
\lim_{n\rightarrow \infty} \frac{\langle D(x) \rangle_n}{n^{1+3\alpha}}=\frac{ \xi^2 \sigma^2}{(1+3\alpha)g(x)}.
\end{equation}
Therefore, we deduce from the strong law of large numbers for martingales
that, as soon as $0<\alpha<1/3$, 
\begin{equation}
\label{ASCVGNNDN}
\lim_{n\rightarrow \infty} \frac{1}{n} M_n^{C}(x) = 0 \hspace{0.3cm} \text{a.s.}
\hspace{1cm}\text{and}\hspace{1cm}
\lim_{n\rightarrow \infty} \frac{1}{n} D_n(x) = 0 \hspace{0.3cm} \text{a.s.}
\end{equation}
Finally, it follows from \eqref{NEWDECOFCPN} together with \eqref{CVGMEANVARCN} and \eqref{ASCVGNNDN} that
for any $x \in \dR$ such that $g(x)>0$,
\begin{equation*}
\lim_{n\rightarrow \infty} \fcnp = f^\prime(x) \hspace{1cm}\text{a.s.}
\end{equation*} 
which achieves the proof of Theorem \ref{T-ASCVG}.
\demend

\vspace{-2ex}

\section{Proofs of the asymptotic normality results.}
\label{AppendixB}
\renewcommand{\thesection}{\Alph{section}}
\renewcommand{\theequation}{\thesection.\arabic{equation}}
\setcounter{section}{2}
\setcounter{equation}{0}

In order to prove Theorem \ref{T-CLT}, we shall make use of the central limit theorem for martingales given e.g. by Theorem 2.1.9 of \cite{duflo1997random}.
First of all, we focus our attention on convergence \eqref{CLTWJD} since
it is the easiest convergence to prove. \vspace{-1ex}\\

\noindent
{\bf Proof of convergence \eqref{CLTWJD}.}
It follows from \eqref{NEWDECOFCPN} that 
$$
\sqrt{nh_n^3}\bigl(\fcnp-f^\prime(x)\bigr) = 
\dfrac{\sqrt{nh_n^3}}{n} \left( M_n^{C}(x)+\dE[C_n(x)]+D_n(x)-n f^\prime(x)\right),
$$
which implies the martingale decomposition
\begin{equation}
\label{DEFOFFCPN_FPN}
\sqrt{nh_n^3}\bigl(\fcnp-f^\prime(x)\bigr) =\dfrac{1}{\sqrt{n^{1+3\alpha}}}\bigl(
\langle  e, \cM_n(x) \rangle + \widecheck{R}_n(x) \bigr)
\end{equation}
where 
\begin{equation*}
 e=
\begin{pmatrix}
\ 1 \ \\ \ 1 \ 
\end{pmatrix},
\hspace{1cm}
\cM_n(x)= 
\begin{pmatrix}
M_n^{C}(x)\\D_n(x)
\end{pmatrix},
\end{equation*}
and the remainder 
\begin{equation}
\label{DEFRESTFCPN}
\widecheck{R}_n(x)=\dE[C_n(x)]-nf^\prime(x)=\skn \bigl(\dE[c_k(x)]-f^\prime(x)\bigr).
\end{equation}
It follows from Taylor's formula that it exists $\theta_f \in ]0,1[$ such that, for any $x \in \dR$,
\begin{align*}
\dE [c_n(x)] & = \int_{\dR} f(x_n) v_n(x_n,x) dx_n =\dfrac{1}{h_n}\int_{\dR} f(x-h_ny)K^\prime(y) dy \\
& = f^\prime(x)+ \dfrac{h_n}{2} \int_{\dR} f^{\prime \prime}(x-h_ny\theta_f)y^2 K^\prime(y) dy,
\end{align*}
where $v_n$ is defined in (\ref{vn}). Since $f^{\prime \prime}$ is bounded, we have
\begin{equation}
\label{CONTRESTFCPN}
\sup_{x \in \dR}\bigl|\dE[c_n(x)]-f^\prime(x)\bigr| \leq M_f \tau^2 h_n
\end{equation}
where 
$$
M_f={\displaystyle \sup_{x \in \dR}} | f^{\prime \prime}(x) |
\hspace{1cm}\text{and} \hspace{1cm}
\tau^2= \frac{1}{2} \int_{\dR} y^2 \bigl| K^\prime(y) \bigr| dy.
$$
Hence, we deduce from \eqref{DEFRESTFCPN} and \eqref{CONTRESTFCPN} that
$$
\sup_{x \in \dR}\bigl|\widecheck{R}_n(x)\bigr| \leq  \tau^2 M_f  \sum_{k=1}^n h_k.
$$
However, it is easily seen that
$$
\sum_{k=1}^n h_k \leq \frac{1}{1-\alpha}n^{1-\alpha}.
$$
Therefore, as soon as $\alpha >1/5$, we obtain that
\begin{equation}
\sup_{x \in \dR} \bigl| \widecheck R_n(x) \bigr|=o(\sqrt{n^{1+3\alpha}}).
\label{REST_FCPN}
\end{equation}
Hereafter, the predictable quadratic variation (\cite{duflo1997random})
of the two-dimensional real martingale $(\cM_n(x))$ is given, for all $n \geq 1$, by the diagonal matrix
\begin{equation*}
\langle \cM(x) \rangle_n=
\begin{pmatrix}
\langle M^{C}(x) \rangle_n & 0 \\
0 & \langle D(x) \rangle_n
\end{pmatrix}.
\end{equation*}
Then, it follows from \eqref{CVGQVNNDN} that for any $x \in \dR$ such that $g(x)>0$,
\begin{equation}
\label{CVGQVMARTVECTNN}
\lim_{n\rightarrow \infty} \dfrac{1}{n^{1+3\alpha}}\langle \cM(x) \rangle_n=
\frac{\xi^2}{(1+3\alpha)g(x)}
\begin{pmatrix}
f^2(x) & 0\\
0 & \sigma^2
\end{pmatrix}.
\end{equation}
Furthermore, it is not hard to see that the martingale $(\cM_n(x))$ satisfies the Lindeberg condition.
As a matter of fact, we have assumed that the sequence $(\veps_n)$ has a finite
moment of order $p>2$. Let $a>0$ be such that $p=2(1+a)$. If we denote
$\Delta \cM_n(x)=\cM_n(x)-\cM_{n-1}(x)$, we have for all $n \geq 1$,
\begin{align}
\dE \bigl [ \|\Delta \cM_n(x)\|^{p} |\cF_{n-1} \bigr ]
&=\dE \bigl [ \bigl( \bigl(\Delta M_n^{C}(x)\bigr)^2+\bigl(\Delta D_n(x)\bigr)^2\bigr)^{1+a} |\cF_{n-1} \bigr ]\notag\\
&\leq  2^{a} \,\dE \bigl [  \bigl|\Delta M_n^{C}(x)\bigr|^{p}+\bigl| \Delta D_n(x)\bigr|^{p} \bigr) |\cF_{n-1} \bigr ].
\label{LIND_DELTN}
\end{align}
On the one hand,
\begin{align}
\dE\bigl [ \bigl|\Delta M_n^{C}(x)\bigr|^{p} |\cF_{n-1}\bigr ] &= \dE\bigl [ |c_n(x)-\dE [c_n(x)]|^{p} |\cF_{n-1}\bigr ]\notag\\
&\leq 2^{p-1} \bigl(\dE\bigl [  \bigl| c_n(x) \bigr|^{p} \bigr ] + \bigl| \dE[c_n(x)] \bigr|^{p} \bigr).
\label{LIND_DELT_martN}
\end{align}
However, as $f^\prime$ is bounded, it follows from \eqref{CONTRESTFCPN} that
$$
\sup_{x \in \dR} \bigl|\dE[c_n(x)]\bigr| \leq m_f+ M_f \tau^2 
$$
where $m_f={\displaystyle \sup_{x \in \dR}} | f^{\prime}(x) |$.
Consequently, it exists a positive constant $C_p$ such that 
\begin{equation}
\label{LIND_cN}
\sup_{x \in \dR} \bigl|\dE[c_n(x)]\bigr|^p \leq C_p. 
\end{equation}
Moreover, 
\begin{align*}
\dE\bigl [  \bigl| c_n(x) \bigr|^{p} \bigr ]
&= \int_{\dR}\dfrac{f(x_n)^{p}}{g(x_n)^{p-1}}\bigl| v_n(x_{n},x) \bigr|^{p}dx_n \\
&=\dfrac{1}{h_n^{2p-1}}\int_{\dR}\dfrac{f(x-h_n y)^{p}}{g(x-h_ny)^{p-1}}\bigl| K^{\prime}(y) \bigr|^{p}dy.
\end{align*}
Hence, for any $x \in \dR$ such that $g(x)>0$, it exist a positive constant $c_p$ such that 
\begin{equation}
\label{LIND_CN}
 \dE\bigl [ \bigl|c_n(x)\bigr|^{p} \bigr ] \leq \dfrac{c_p}{h_n^{2p-1}}.
\end{equation}
Therefore, we deduce from \eqref{LIND_DELT_martN} together with \eqref{LIND_cN} and \eqref{LIND_CN} that
for any $x \in \dR$ such that $g(x)>0$,
\begin{equation}
\label{LIND_majorN}
\dE\bigl [ \bigl|\Delta M_n^{C}(x)\bigr|^{p} |\cF_{n-1}\bigr ]
\leq 2^{p-1}\Bigl(
\dfrac{c_p}{h_n^{2p-1}}+C_p \Bigr).
\end{equation}
On the other hand, we have
$$
\dE\bigl [ \bigl|\Delta D_n(x)\bigr|^{p} |\cF_{n-1}\bigr ] = \dE\bigl [ |d_n(x)|^{p} |\cF_{n-1}\bigr ]
= \dE\bigl [ |\veps_n|^{p} |w_n(X_n,x)|^{p} |\cF_{n-1}\bigr ]
$$
where
$$w_n(X_n,x)= \dfrac{v_n(X_n,x)}{g(X_n)} =\dfrac{1}{h_n^2 g(X_n)} \KpXn.$$
We infer from assumption $(\mathcal{A}_3)$ that
\begin{equation}
\label{LIND_SEPDN}
\dE\bigl [ \bigl|\Delta D_n(x)\bigr|^{p} |\cF_{n-1}\bigr ] 
= \dE\bigl [ |\veps_n|^{p}] \dE\bigl [ |w_n(X_n,x)|^{p}].
\end{equation}
However, the sequence $(\veps_n)$ has a finite moment of order $p>2$ which means that 
it exists a positive constant $E_p$ such that $E_p= \dE\bigl [ |\veps_n|^{p}]$.
Moreover, as in the proof of \eqref{LIND_CN}, we obtain that for any $x \in \dR$ such that $g(x)>0$, 
it exist a positive constant $w_p$ such that 
\begin{equation}
\label{LIND_DN}
 \dE\bigl [ \bigl|w_n(X_n,x)\bigr|^{p} \bigr ] \leq \dfrac{w_p}{h_n^{2p-1}}.
\end{equation}
Hence, it follows from \eqref{LIND_SEPDN} and \eqref{LIND_DN} that
for any $x \in \dR$ such that $g(x)>0$,
\begin{equation}
\label{LIND_majorD}
\dE\bigl [ \bigl|\Delta D_n(x)\bigr|^{p} |\cF_{n-1}\bigr ]
\leq 
\dfrac{E_p w_p}{h_n^{2p-1}}.
\end{equation}
Consequently, we deduce from \eqref{LIND_DELTN} together with
\eqref{LIND_majorN} and \eqref{LIND_majorD} that for any $x \in \dR$ such that $g(x)>0$,
one can find a positive constant $M_p$ such that, for all $n \geq 1$,
\begin{equation}
\label{LIND_majorfinN}
\dE \bigl [ \|\Delta \cM_n(x)\|^{p} |\cF_{n-1} \bigr ] \leq \frac{M_p}{h_n^{2p-1}}
\hspace{1cm}\text{a.s.}
\end{equation}
We recall that $p=2(1+a)$. For any $\veps>0$, if
$\cA_k(x,\veps,n)=\bigl\{ \|\Delta \cM_k(x)\| \geq \veps \sqrt{n^{1+3 \alpha}}\bigr\}$, 
we have from \eqref{LIND_majorfinN},
\begin{align*}
\dfrac{1}{n^{1+3\alpha}} \skn \dE \bigl [ \|\Delta \cM_k(x)\|^{2}
 \rI_{\cA_k(x,\veps, n)}
|\cF_{k-1} \bigr ]
& \leq \frac{1}{\veps^{p-2} n^{b}} \skn \dE \bigl [ \|\Delta \cM_k(x)\|^{p} |\cF_{k-1} \bigr ]  \\
& \leq \frac{M_p}{\veps^{p-2} n^{b}} \skn \frac{1}{h_k^{2p-1}}  \hspace{1cm}\text{a.s.} \\
& \leq \frac{M_p n^c}{\veps^{p-2}} \hspace{1cm}\text{a.s.}
\end{align*}
where
$b=(a+1)(1+3\alpha)$ and $c=a(\alpha -1)$. Since $c<0$, the Lindeberg condition is clearly satisfied.
Finally, we can conclude from the central limit theorem for martingales (\cite{duflo1997random}) that
for any $x \in \dR$ such that $g(x)>0$,
\begin{equation}
\dfrac{1}{\sqrt{n^{1+3\alpha}}}\cM_n(x) \liml \cN \bigl(0,\Gamma(x) \bigl),
\label{CLTNN}
\end{equation} 
where
$$
\Gamma(x)=\frac{\xi^2}{(1+3\alpha)g(x)}
\begin{pmatrix}
f^2(x) & 0\\
0 & \sigma^2
\end{pmatrix}.
$$
Hence, \eqref{DEFOFFCPN_FPN} together with \eqref{REST_FCPN} and \eqref{CLTNN} immediately leads to
$$
\vspace{-2ex}
\sqrt{nh_n^3}\bigl(\fcnp-f^\prime(x)\bigr)  \liml \cN \Bigl(0,\dfrac{\xi^2}{1+3\alpha}\dfrac{f^2(x)+\sigma^2}{g(x)}\Bigr).
$$
\demend
\noindent
{\bf Proof of convergence \eqref{CLTJD}.} It follows from \eqref{DEFHGHAT} that
\begin{equation}
\label{DECHN}
n \hn = P_n(x)+Q_n(x)=M_n^{P}(x) + \dE[P_n(x)] + Q_n(x)
\end{equation}
where $M_n^{P}(x)=P_n(x)-\dE[P_n(x)]$,
 \begin{align*}
P_n(x) &= \skn p_k(x)= \skn f(X_k)u_k(X_k,x), \\
Q_n(x) &= \skn q_k(x)= \skn \veps_k u_k(X_k,x) 
\end{align*}
with
$$u_n(X_n,x)= \dfrac{1}{h_n} \KXn.$$
Hence, for any $x \in \dR$ such that $g(x)>0$, we obtain from \eqref{JD}, \eqref{NEWDECOHN} and \eqref{DECHN} 
\begin{align*}
n(\ftnp - f^{\prime}(x))&=\dfrac{1}{g(x)}\bigl(M_n^{\!A}(x)+\dE[A_n(x)]+B_n(x)\bigr) \\
&\hspace{1cm}-\dfrac{g^{\prime}(x)}{g^2(x)}\bigl(M_n^{\!P}(x)+\dE[P_n(x)]+Q_n(x)\bigr) - nf^{\prime}(x)
\end{align*}
which leads to the martingale decomposition
\begin{equation}
\label{DEFOFFTPN_FPN}
\sqrt{nh_n^3}\bigl(\ftnp-f^\prime(x)\bigr) =\dfrac{1}{\sqrt{n^{1+3\alpha}}}\bigl(
\langle \widetilde e(x), \cM_n(x) \rangle +\widetilde R_n(x) \bigr)
\end{equation}
where 
\begin{equation*}
\widetilde e(x)=
\frac{1}{g^2(x)}
\begin{pmatrix}
\ g(x) \ \\ \ g(x) \ \\ -g^{\prime}(x) \ \\ -g^{\prime}(x) \
\end{pmatrix},
\hspace{1cm}
\cM_n(x)= 
\begin{pmatrix}
M_n^{\!A}(x)\\B_n(x)\\M^{\!P}_n(x)\\Q_n(x)
\end{pmatrix},
\end{equation*}
and the remainder 
$$\widetilde R_n(x)=\dfrac{1}{g(x)}\dE[A_n(x)]-\dfrac{g^\prime(x)}{g^2(x)}\dE[P_n(x)]-nf^\prime(x).$$
We saw in  \eqref{MEANAN} that $\dE[a_n(x)]=\bigl(f(x)g(x)\bigr)^\prime+R_n(x)$ 
where $\sup_{x \in \dR}\bigl| R_n(x)\bigr|=O(h_n)$. By the same token, $\dE[P_n(x)]=f(x)g(x)+\zeta_n(x)$ 
where $\sup_{x \in \dR} \bigl| \zeta_n(x)\bigr|=O(h_n^2)$. 
Therefore, as soon as $\alpha >1/5$, we obtain that
\begin{equation}
\sup_{x \in \dR} \bigl| \widetilde R_n(x) \bigr|=O \left (\skn h_k \right )=o(\sqrt{n^{1+3\alpha}}).
\label{REST_FTPN}
\end{equation}
Furthermore, as in the proof of \eqref{CVGQVMNBN} and \eqref{CVGQVMARTVECTNN}, the predictable quadratic variation  
of the four-dimensional real martingale $(\cM_n(x))$
%
%
satisfies, for any $x \in \dR$, 
\begin{equation}
\label{CVGQVMARTVECTMM}
\lim_{n\rightarrow \infty} \dfrac{1}{n^{1+3\alpha}}\langle \cM(x) \rangle_n= \Gamma(x)
\end{equation}
where $\Gamma(x)$ is the four-dimensional covariance matrix given by
$$
\Gamma(x)=\frac{\xi^2 g(x)}{(1+3\alpha)}
\begin{pmatrix}
f^2(x) & 0 & \ 0 & \ 0 \ \\
0 & \sigma^2 & \ 0 & \ 0 \ \\
0 & 0 & \ 0 & \ 0 \ \\
0 & 0 & \ 0 & \ 0 \
\end{pmatrix}.
$$
Moreover, via the same lines as in the proof of \eqref{LIND_majorfinN}, we can also show 
that $(\cM_n(x))$ satisfies the Lindeberg condition. Finally, we find from the central limit theorem 
for martingales (\cite{duflo1997random}) that for any $x \in \dR$,
\begin{equation*}
\dfrac{1}{\sqrt{n^{1+3\alpha}}}\cM_n(x) \liml \cN \bigl(0,\Gamma(x) \bigl),
\end{equation*} 
which implies, from \eqref{DEFOFFTPN_FPN} and \eqref{REST_FTPN}, that for any $x \in \dR$ such that $g(x)>0$,
$$
\vspace{-2ex}
\sqrt{nh_n^3}\bigl(\ftnp-f^\prime(x)\bigr)  \liml \cN \left (0,\dfrac{\xi^2}{1+3\alpha}\dfrac{f^2(x)+\sigma^2}{g(x)}\right ).
$$
\demend
\noindent
{\bf Proof of convergence \eqref{CLTNWD}.}
First of all, for any $x \in \dR$, denote $h(x)=f(x)g(x)$. It follows from 
\eqref{DEFFHAT} and \eqref{DEFHGHAT} that for any $x \in \dR$ such that $g(x)>0$,
\begin{align}
\fn-f(x)&=\dfrac{\hn}{\gn}-\dfrac{h(x)}{g(x)}=\dfrac{1}{ g(x) \gn}\Bigl( \hn g(x) -h(x) \gn\Bigr) \notag \\
&=\dfrac{1}{g(x)\gn }\Bigl(g(x)\Bigl(\hn-h(x)\Bigr)-h(x)\Bigl(\gn-g(x)\Bigr)\Bigr) \notag \\
&=\dfrac{1}{\gn }\Bigl(\hn-h(x)\Bigr)-\dfrac{h(x)}{ g(x)\gn }\Bigl(\gn-g(x)\Bigr).
\label{DECCLTFHAT}
\end{align}
By the same token, we obtain from \eqref{NWD} together with tedious but straightforward calculation that
for any $x \in \dR$ such that $g(x)>0$,
\begin{align}
\fnp - f^{\prime}(x)&=\Bigl( \dfrac{\hnp}{\gn}-\dfrac{\hn \gnp}{\gnc}\Bigr)-\Bigl(\dfrac{h'(x)}{g(x)}-\dfrac{h(x)g'(x)}{g^2(x)}\Bigr) \notag\\
&=\dfrac{1}{\gn}\Bigl( \hnp-h'(x) \Bigr)-\dfrac{f^{\prime}(x)}{\gn}\Bigl( \gn-g(x) \Bigr)\notag \\
&\hspace{1cm}- \dfrac{\hn}{\gnc}\Bigl(\gnp -g^{\prime}(x)\Bigr)- \dfrac{g^{\prime}(x)}{\gn}\Bigl(\fn -f(x)\Bigr).
\label{DECCLTFHATP}
\end{align}
Hence, we obtain from \eqref{DECCLTFHAT} and \eqref{DECCLTFHATP} that
for any $x \in \dR$ such that $g(x)>0$,
\begin{align}
\fnp - f^{\prime}(x)&= \dfrac{1}{\gn}\Bigl( \hnp-h'(x) \Bigr) + \dfrac{\widehat{\ell}_n(x)}{\gnc}\Bigl( \gn-g(x) \Bigr)\notag \\
&\hspace{1cm} -\dfrac{\hn}{\gnc}\Bigl(\gnp -g'(x)\Bigr) - \dfrac{g^{\prime}(x)}{\gnc}\Bigl(\hn -h(x)\Bigr)
\label{DECCLTFHATPFIN}
\end{align}
where $\widehat{\ell}_n(x)=f(x)g^{\prime}(x)-f^{\prime}(x)\gn$.
Therefore, we deduce from identity \eqref{DECCLTFHATPFIN} the martingale decomposition
\begin{equation}
\label{DECMARTCLT_FPN}
\sqrt{nh_n^3}\bigl(\fnp-f^\prime(x)\bigr) =\dfrac{1}{\sqrt{n^{1+3\alpha}}}\bigl(
\langle \widehat e_n(x), \cM_n(x) \rangle +\widehat R_n(x) \bigr)
\end{equation}
with
\begin{equation*}
\widehat e_n(x)=
\frac{1}{\gnc}
\begin{pmatrix}
\ \gn \ \\ \ \gn \ \\ - \hn \ \\ \widehat{\ell}_n(x) \\ -g^{\prime}(x) \ \\ -g^{\prime}(x) \
\end{pmatrix},
\hspace{1cm}
\cM_n(x)= 
\begin{pmatrix}
M_n^{\!A}(x)\\B_n(x)\\ M^{V}_n(x)\\ M^{U}_n(x)\\M^{\!P}_n(x)\\Q_n(x)
\end{pmatrix},
\end{equation*}
where the martingale difference sequences $(M_n^{\!A}(x))$, $(B_n(x))$ and $(M_n^{\!P}(x))$, $(Q_n(x))$
were previously defined in \eqref{NEWDECOHN} and \eqref{DECHN}, while
the martingale difference sequences $(M_n^{U}(x))$ and $(M_n^{V}(x))$
are given by $M_n^{U}(x)=U_n(x)-\dE[U_n(x)]$ and $M_n^{V}(x)=V_n(x)-\dE[V_n(x)]$ with
\begin{align*}
U_n(x) &= \skn u_k(X_k,x)=\skn\dfrac{1}{h_n} \KXn, \\
V_n(x) &= \skn v_k(X_k,x)=\skn\dfrac{1}{h_n^2} \KpXn.
\end{align*}
It is not hard to see that the remainder $\widehat R_n(x)$, which can be explicitely calculated, plays a negligible role
since, as soon as $\alpha>1/5$,
\begin{equation}
\sup_{x \in \dR} \bigl| \widehat R_n(x) \bigr|=o(\sqrt{n^{1+3\alpha}}).
\label{REST_FHATN}
\end{equation}
It remains to establish the asymptotic behavior of the six-dimensional real martingale $(\cM_n(x))$. As in the proof of
\eqref{CVGQVMARTVECTNN} and \eqref{CVGQVMARTVECTMM}, we can show that for any $x \in \dR$, 
\begin{equation}
\label{CVGQVMARTVECTCLT}
\lim_{n\rightarrow \infty} \dfrac{1}{n^{1+3\alpha}}\langle \cM(x) \rangle_n= \Gamma(x)
\end{equation}
where $\Gamma(x)$ is the six-dimensional covariance matrix given by
$$
\Gamma(x)=
\frac{\xi^2 g(x)}{(1+3\alpha)}
\begin{pmatrix}
f^2(x) & 0 & f(x) &\ 0 & \ 0 & \ 0\ \\
0 & \sigma^2 & \ 0 & \ 0 & \ 0 & \ 0 \ \\
f(x) & 0 & 1 &\ 0 & \ 0 & \ 0 \ \\
0 & 0 & \ 0 & \ 0 & \ 0 & \ 0 \ \\
0 & 0 & \ 0 & \ 0 & \ 0 & \ 0 \ \\
0 & 0 & \ 0 & \ 0 & \ 0 & \ 0 \
\end{pmatrix}.
$$
Moreover, via the same lines as in the proof of \eqref{LIND_majorfinN}, $(\cM_n(x))$ satisfies the Lindeberg condition.
Hence, we obtain from the central limit theorem 
for martingales (\cite{duflo1997random}) that for any $x \in \dR$,
\begin{equation}
\dfrac{1}{\sqrt{n^{1+3\alpha}}}\cM_n(x) \liml \cN \bigl(0,\Gamma(x) \bigl).
\label{CLTMM}
\end{equation}
Furthermore, it follows from Lemma \ref{L-ASCVG} that for any $x \in \dR$ such that $g(x)>0$,
$\widehat e_n(x)$ converges a.s. to $e(x)$ where
\begin{equation}
\label{ASCVGEHATN}
e(x)=
\frac{1}{g^2(x)}
\begin{pmatrix}
\ g(x) \ \\ \ g(x) \ \\ -h(x) \ \\ \ \ell(x) \ \\ -g^{\prime}(x) \ \\ -g^{\prime}(x) \
\end{pmatrix}
\end{equation}
with
$\ell(x)=f(x)g^{\prime}(x)-f^{\prime}(x)g(x)$. Finally, we deduce from
\eqref{DECMARTCLT_FPN}, \eqref{REST_FHATN}, \eqref{CLTMM} and \eqref{ASCVGEHATN}
together with Slutsky's lemma that for any $x \in \dR$ such that $g(x)>0$,
\begin{equation}
\label{CLTFHATNPFINAL}
\sqrt{nh_n^3}\bigl(\fnp-f^\prime(x)\bigr)  \liml \cN \bigl(0,\sigma^2(x)\bigr)
\end{equation}
where $\sigma^2(x)= \langle e(x), \Gamma(x) e(x) \rangle$. However, as
$h(x)=f(x)g(x)$,
it is not hard to see that
\begin{align*}
\sigma^2(x)&=
\dfrac{\xi^2}{(1+3\alpha)g^{3}(x)} 
\begin{pmatrix}
\ g(x) \ \\ \ g(x) \ \\  -f(x)g(x) \
\end{pmatrix}^T
\begin{pmatrix}
f^2(x) & 0 & f(x) \\
0 & \sigma^2 & \ 0 \ \\
f(x) & 0 & 1 
\end{pmatrix} 
\begin{pmatrix}
\ g(x) \ \\ \ g(x) \ \\ -f(x)g(x) \
\end{pmatrix} \\
&=
 \dfrac{\xi^2}{(1+3\alpha)g(x)} 
\begin{pmatrix}
\ 1 \ \\ \ 1 \ \\  -f(x) \
\end{pmatrix}^T
\begin{pmatrix}
f^2(x) & 0 & f(x) \\
0 & \sigma^2 & \ 0 \ \\
f(x) & 0 & 1 
\end{pmatrix} 
\begin{pmatrix}
\ 1 \ \\ \ 1 \ \\ -f(x) \
\end{pmatrix} \\
&=\dfrac{\xi^2 \ \sigma^2}{(1+3\alpha)g(x)},
\end{align*}
which completes the proof of Theorem \ref{T-CLT}.
\demend

\noindent
{\bf Acknowledgements.}
This work was supported by the ASPEET Grant from the Universit\'e Bretagne Sud and the Centre National de la Recherche Scientifique. 
We would like to thanks J. C.  Massabuau, D. Tran, P. Ciret and M. Sow, from Bordeaux University, for many 
fruitful discussions.

\bibliographystyle{agsm}

\bibliography{MS_Bercu_Capderou_Durrieu}

\end{document}